\documentclass[11pt]{amsart}

\usepackage{amsfonts}
\usepackage{amscd}
\usepackage{amsbsy}
\usepackage{epsfig}
\usepackage{epic,eepic}
\usepackage{graphicx}
\usepackage[all]{xy}
\usepackage{amsmath}
\usepackage{amsthm}
\usepackage{setspace}
\usepackage{url}
\usepackage{tikz}
\usetikzlibrary{patterns}
\usepackage{here}
\usepackage{caption}
\usepackage{subcaption}
\usepackage{hyperref}

\theoremstyle{plain}
\newtheorem*{theorem*}{Theorem}
\newtheorem*{lemma*}{Lemma}
\newtheorem*{proposition*}{Proposition}
\newtheorem*{conjecture*}{Conjecture}
\newtheorem*{corollary*}{Corollary}
\newtheorem*{problem*}{Problem}
\newtheorem{theorem}{Theorem}[section]
\newtheorem{lemma}[theorem]{Lemma}
\newtheorem{proposition}[theorem]{Proposition}
\newtheorem{corollary}[theorem]{Corollary}

\newtheorem{conjecture}[theorem]{Conjecture}

\theoremstyle{definition}

\newtheorem{definition}[theorem]{Definition}

\newcommand{\bbA}{\mathbb A}

\newcommand{\PP}{\mathbb P}
\newcommand{\QQ}{\mathbb Q}

\newcommand{\CC}{\mathbb C}

\newcommand{\OO}{\mathcal O}

\newcommand{\RR}{\mathbb R}

\newcommand{\Nef}[1]{\operatorname{Nef}({#1})^e}
\newcommand{\Mov}[1]{\operatorname{Mov}({#1})^e}
\newcommand{\nef}[1]{\operatorname{Nef}({#1})}
\newcommand{\mov}[1]{\operatorname{Mov}({#1})}

\newcommand{\Pic}{\operatorname{Pic}}

\newcommand{\bbQ}{\mathbb Q}
\newcommand{\bbR}{\mathbb R}

\newcommand{\iso}{\cong}
\newcommand{\Aut}{\operatorname{Aut}}
\newcommand{\PsAut}{\operatorname{PsAut}}

\newcommand{\Curv}{\operatorname{Curv}}
\newcommand{\arrow}{\rightarrow}
\newcommand{\ZZ}{\mathbb{Z}}
\newcommand{\Bl}{\operatorname{Bl}}
\newcommand{\Id}{\operatorname{Id}}

\begin{document}
\title{Fano manifolds of index $n-2$ and the cone conjecture}

\author{Izzet Coskun}
\email{coskun@math.uic.edu}
\address{Department of Mathematics, Statistics and CS \\University of Illinois at Chicago, Chicago, IL 60607}

\author{Artie Prendergast-Smith}
\email{A.Prendergast-Smith@lboro.ac.uk}
\address{Department of Mathematical Sciences, Loughborough University LE11 3TU UK}

\thanks{During the preparation of this article the first author was partially supported by the NSF CAREER grant DMS-0950951535; the second author was partially supported by EPSRC First Grant EP/L026570/1.}

\begin{abstract}
The Morrison--Kawamata Cone Conjecture predicts that the action of the automorphism group on the effective nef cone and the action of the pseudo-automorphism group on the effective movable cone of a klt Calabi-Yau pair have rational, polyhedral fundamental domains. In \cite{CPS}, we proved the conjecture for certain blowups of Fano manifolds of index $n-1$. In this paper, we consider the Morrison--Kawamata conjecture for blowups of Fano manifolds of index $n-2$.  
\end{abstract}

\maketitle

\section{Introduction}
Cones of divisors play a central role in birational geometry. Consequently, it is a basic issue to understand these cones for important classes of algebraic varieties. By the Cone Theorem of Mori-Kawamata-Koll\'{a}r-Reid-Shokurov \cite[Theorem 3.7]{KollarMori}, the nef cone of a Fano manifold is rational polyhedral. By Birkar--Cascini--Hacon--McKernan \cite{BCHM}, the effective and movable cones of a Fano manifold are also rational polyhedral. In contrast, as soon as the anti-canonical bundle $-K_X$ is not ample, these cones may have infinitely many extremal rays. For example, the nef cone of a $K3$ surface with infinitely many  $(-2)$-curves is not polyhedral. 

Nevertheless, the Morrison--Kawamata conjecture \cite{Morrison1992,Kawamata1997,Totaro2008} predicts that for a Calabi-Yau manifold (or more generally a klt Calabi-Yau pair), the effective nef and the effective movable cones have rational, polyhedral fundamental domains for the action of the automorphism and the pseudo-automorphism groups, respectively. (See \S \ref{sect-intro} for the precise statement.) Despite significant progress in many cases, the conjecture remains  open, even for Calabi--Yau threefolds. The essential difficulty is producing enough automorphisms on a variety with complicated birational geometry.

In this paper, we study certain blowups of Fano manifolds of index $n-2$ and verify the Morrison--Kawamata conjecture for them. For $n \geq 4$, let $Z$ be an $n$-dimensional Fano manifold of index $n-2$ over an algebraically closed field of characteristic zero. These manifolds are classified (\cite{IP}, \cite{Mukai}, see \S \ref{sect-fano}). Write $-K_Z = (n-2)H$ for an ample divisor $H$ and let $V \subset |H|$ be a general $(n-2)$-dimensional general linear system with base locus $C$. Let $X$ be the blowup of $Z$ along $C$. Then $|-K_X|$ is a basepoint-free linear system that defines a $K3$ fibration over $\PP^{n-2}$. By Bertini's Theorem, we can take a smooth divisor $D$ in the linear system $|-2K_X|$ and set $\Delta= \frac{1}{2} D$ . Then the pair $(X, \Delta)$ is a klt Calabi-Yau pair and the Morrison--Kawamata conjecture applies. Our first result is the following.

\begin{theorem*}[\ref{thm-nef}]
The nef cone $\nef {X}$ is contained in the effective cone and is a rational, polyhedral cone. In particular, the Morrison--Kawamata conjecture on the effective nef cone holds for $X$.
\end{theorem*}

We next study the conjecture for movable cones. We resolve the conjecture in several interesting cases. It is easy to show (see Corollary \ref{remark}) that the conjecture is true when $\rho(Z)=1$, so we concentrate on the higher Picard rank case. We show the following.

\begin{theorem*}[\S \ref{sect-finite}, \ref{sect-infinite}]
Let $Z$ be one of the following Fano manifolds:
\begin{enumerate}
 \item $\PP^1 \times \PP^3$
\item the blowup of a cone over a quadric threefold at the vertex
\item  $\PP^1 \times V$, where $V$ is a del Pezzo manifold of Picard rank one 
\item $\PP^1 \times \PP^1 \times \PP^1 \times \PP^1$
\item the double cover of $\PP^2 \times \PP^2$ branched over a divisor of type $(2,2)$
\item $\PP^3 \times \PP^3$, or a bi-linear section of it, namely the flag variety $F(1,3;4)$  or the intersection of two divisors of type $(1,1)$.
\end{enumerate}
Let $-K_Z=(n-2)H$ and $V \subset |H|$ a linear system that is \textbf{general} in cases 1--4 or \textbf{very general} in cases 5 and 6. Let $X$ be the blowup of $Z$ along the base locus of $V$. Then the Morrison--Kawamata Cone Conjecture on the effective movable cone holds for $X$. In the first three cases the movable cone has finitely many faces; in the last two it has infinitely many faces.
\end{theorem*}
~\newline
In  \cite{CPS}, we proved the Morrison--Kawamata conjecture for certain blowups of Fano manifolds of index $n-1$. These manifolds have elliptic fibrations and the group structure on the generic fiber provides a source of birational automorphisms that suffice to verify the conjecture. In the present case, we use automorphisms of the $K3$ fibers, guaranteed to exist by the Global Torelli Theorem, to find birational automorphisms of our varieties. However, the implementation is much more complicated for two reasons. First, the base locus of the linear series $V$ is a curve rather than finitely many points, which makes understanding effective and movable divisor classes much more difficult. Second, the fibers are $K3$ surfaces, which unlike curves can have complicated cones of divisors and automorphism groups.

In spite of these difficulties, we are able to verify the conjecture in the stated cases. Interestingly, the different cases exhibit radically different behavior.  In some cases, the $K3$ fibers have finite automorphism group, and the birational geometry of the total space is simple.  In others, the $K3$ surface has infinite, even highly non-abelian automorphism group, and the birational geometry of the total space is correspondingly complicated.

\subsection*{Organization of the paper}  In \S \ref{sect-intro} we give a precise statement of the Morrison--Kawamata conjecture. In \S \ref{sect-fano} we recall Mukai's classification of Fano manifolds of index $n-2$, and analyze the fundamental linear system on such manifolds. This analysis is used in \S \ref{sect-nef} to give a  proof of the conjecture for the nef cone in all cases. In \S \ref{sect-generalities}, we collect some general results needed to study the conjecture for the movable cone. In particular, we discuss Looijenga's Theorem \ref{thm-looijenga} on fundamental domains, and the ``Lifting" Theorem \ref{theorem-lifting} allowing us to pass from the nef cone of the general fiber to the movable cone of the total space. In \S \ref{sect-finite}, we study the conjecture for $\PP^1 \times \PP^3$, the blowup of the cone over a quadric threefold, and $\PP^1 \times V$, where $V$ is a del Pezzo manifold of Picard rank one.  In \S \ref{sect-infinite}, we study the branched double cover of $\PP^2 \times \PP^2$, $\PP^3 \times \PP^3$ and its bi-linear sections, and $\PP^1 \times \PP^1 \times \PP^1 \times \PP^1$.

\subsection*{Acknowledgments} We thank Lawrence Ein and Burt Totaro for fruitful discussion.  We are especially grateful to Chris Brav, who generously shared his Sage code to compute fundamental domains. 

Part of this work was done while the second author was a participant in the Hausdorff Trimester Program ``Birational and hyperk\"ahler geometry". He is extremely grateful to the Hausdorff Institute for the invitation and hospitality.

\section{The Morrison--Kawamata conjecture} \label{sect-intro}
In this section, we recall the precise statement of the Morrison--Kawamata Cone Conjecture.

A {\em rational polyhedral} cone in a real vector space with a $\bbQ$-structure is a closed, convex cone generated by finitely many rational extremal rays. A finite rational polyhedral cone $\Pi$ is a {\em fundamental
  domain} for the action of a group $G$ on a cone $C$ if $C = G \cdot \Pi$ and the interiors of $\Pi$ and $g \Pi$ are disjoint for $1 \not= g \in G$.

Let $X$ be a projective variety. A line bundle $L$ on $X$ is {\em nef} if $L \cdot C \geq 0$ for every proper curve $C$ on $X$.  It is  {\em movable} if its base locus has codimension at least 2 in $X$. 
Let $N^1(X)$ denote the N\'{e}ron-Severi space of  $X$. The classes of nef line bundles generate a closed convex cone in $N^1(X)$ called the {\em nef cone} $\nef{X}$. Similarly, we denote by $\mov{X}$ the smallest {\em closed} cone containing the classes of all movable line bundles. Let $\Nef{X}$ and $\Mov{X}$ denote the {\em effective nef cone} and the {\em effective movable cone},  which are the subcones of  the nef and movable cones generated by effective Cartier divisors, respectively. 

A {\em pseudo-isomorphism} from $X_1$ to $X_2$ is a birational map which is an isomorphism in codimension one. A {\em small $\bbQ$-factorial modification} (SQM) of  $X$ is a pseudo-isomorphism from $X$ to another $\bbQ$-factorial projective variety $X'$. If $\alpha: X \dashrightarrow X'$ is an SQM, then pullback of divisors gives an identification $N^1(X') \iso N^1(X)$ under which $\Nef{X'}$ is identified with a subcone of $\Mov{X}$. We abuse notation and also denote this subcone by $\Nef{X'}$.

Let $\Delta$ be an $\bbR$-divisor on a normal $\bbQ$-factorial variety $X$. Then the pair 
 $(X,\Delta)$ is {\it klt} if, for all resolutions $\pi:
\tilde{X} \arrow X$ with a simple normal crossing $\RR$-divisor
$\tilde{\Delta}$ such that $K_{\tilde{X}}+\tilde{\Delta} =
\pi^*(K_X+\Delta)$, the coefficients of $\tilde{\Delta}$ are less than
1. In particular, if $X$ is smooth and $D$ is a smooth divisor on $X$,
then $(X,rD)$ is klt for any $r<1$. The pair $(X,\Delta)$ is a
{\it klt Calabi--Yau pair} if $(X,\Delta)$ is a $\QQ$-factorial klt
pair with $\Delta$ effective such that $K_X+\Delta$ is numerically
trivial. 

Let $\Aut(X,\Delta)$, respectively
 $\PsAut(X,\Delta)$, denote the group of automorphisms, respectively pseudo-automorphisms,  of $X$
which preserve a divisor $\Delta$. These groups act on $N^1(X)$ preserving the cones $\Nef{X}$ and $\Mov{X}$, respectively. Moreover, the action of $\PsAut(X,\Delta)$ permutes the effective nef cones $\Nef{X'}$ of SQMs inside the movable cone \cite[Lemma 1.5]{Kawamata1997}. 
Denote the images of these groups in $\operatorname{GL}(N^1(X))$  
 by $\Aut^*(X,\Delta)$ and $\PsAut^*(X,\Delta)$.   With these conventions, the
Morrison--Kawamata cone conjecture states the following.

\begin{conjecture}[Morrison--Kawamata] \label{conj-coneconjecture}
Let $(X, \Delta)$ be a klt Calabi--Yau pair. Then:

\begin{enumerate}
\item The number of $\Aut(X,\Delta)$-equivalence classes of faces of
 $\Nef{X}$ corresponding to birational
contractions or fiber space structures is finite. Moreover, there
exists a  rational polyhedral cone $\Pi$ which is a fundamental
domain for the action of $\Aut^*(X,\Delta)$ on $\Nef{X}$.

\item The number of $\PsAut(X,\Delta)$-equivalence classes of
nef cones $\Nef{X'}$ in the cone $\Mov{X}$
corresponding to SQMs $\alpha: X' \dashrightarrow X$ is finite. Moreover, there
exists a  rational polyhedral cone $\Pi'$ which is a fundamental
domain for the action of $\PsAut^*(X,\Delta)$ on $\Mov{X}$.
\end{enumerate}
\end{conjecture}

The conjecture has been proved for Calabi--Yau surfaces by
Looijenga--Sterk and Namikawa \cite{Sterk1985,Namikawa1985}, for klt
Calabi--Yau pairs of dimension 2 by Totaro \cite{Totaro2008b}, for
Calabi--Yau fiber spaces of dimension 3 over a positive-dimensional
base by Kawamata \cite{Kawamata1997}, and for abelian varieties by the
second author \cite{Prendergast-Smith2010}. For Calabi--Yau 3-folds
there are significant results by Oguiso--Peternell \cite{Oguiso2001},
Szendr\"oi \cite{Szendroi1999}, Uehara \cite{Uehara2004}, and Wilson
\cite{Wilson1992}. The conjecture has been verified for $(2, 2, \dots, 2)$ divisors in $(\PP^1)^n$ by Cantat and Oguiso \cite{cantatoguiso} and for certain blowups of Fano manifolds of index $n-1$ in \cite{CPS}.  However, the full conjecture remains open.  

\section{Fano manifolds of index $n-2$} \label{sect-fano}
Let $n \geq 4$ be an integer.  In this section, we recall the classification of $n$-dimensional Fano manifolds of index $n-2$ for the reader's convenience. We refer the reader to \cite{AC}, \cite{Mukai}, \cite{IP} and \cite{Wisniewski} for more details.

\begin{definition}
A projective manifold $Z$ is called \emph{Fano} if its anticanonical line bundle $\OO_Z(-K_Z)$ is ample. The \emph{index} $i_Z$ of a Fano manifold $Z$ is the largest positive integer $m$ for which there exists a line bundle $A$ on $Z$ such that $\OO_Z(-K_Z) = A^{\otimes m}$. 
\end{definition}
The index of an $n$-dimensional Fano manifold $Z$ is at most $n+1$, with equality if and only if $Z = \PP^n$. Fano manifolds of index $n$ are quadric hypersurfaces \cite{KO}. Fano manifolds of index $n-1$, known as {\em del Pezzo manifolds}, were classified by Fujita  \cite{Fujita}. 


\subsection*{Mukai manifolds} Fano manifolds of dimension $n$ and index $n-2$ are called {\em Mukai manifolds}.  In \cite{Wisniewski}, 
Wi\'{s}niewski  classified Fano $n$-folds $Z$ of index at least $\frac{n+1}{2}$ and $\rho(Z)\geq 2$. His theorem implies that Fano $n$-folds with   $n \geq 5$, $\rho(Z)\geq 2$ and $i_Z=n-2$ are 
$$\PP^3 \times \PP^3, \ \ \PP^2 \times Q^3, \ \ F(1,3; 4), \ \ \Bl_{l}\PP^5,$$
where $Q^3$ is a  quadric threefold, $F(1,3;4)$ is the partial flag variety parameterizing partial flags $F_1 \subset F_3$ in a four-dimensional vector space and $\Bl_l(\PP^5)$ is the blowup of $\PP^5$ along a line. 

The Fano $n$-folds with $\rho(Z)=1$ and $i_z=n-2$ are also classified \cite{Mukai}, \cite{IP}. We will not need the detailed classification since we will see (Remark \ref{remark}) that the Morrison--Kawamata conjecture is easy to verify in this case. For the reader's convenience, we note that the Mukai $n$-folds  with $n\geq 4$ and $\rho(Z)=1$  are the following.
\begin{enumerate}
\item[(M1)] A quadric section of the cone over $G(2,5)$ and its linear sections,
\item[(M2)] Certain homogeneous varieties and their linear sections: 
\begin{enumerate}
\item The spinor variety $OG_{+}(5,10)$, 
\item The Grassmannian $G(2,6)$, 
\item The symplectic Grassmannian $SG(3,6)$, and 
\item The five-dimensional $G_2$ homogeneous variety $G_2/P_2$.
\end{enumerate}
\item[(M3)] Certain complete intersections in projective space and weighted projective spaces:  
\begin{enumerate}
\item Quartic hypersurfaces $X_4 \subset \PP^{n+1}$, 
\item $(2,3)$ complete intersections $X_{2,3} \subset \PP^{n+2}$, 
\item $(2,2,2)$ complete intersections $X_{2,2,2} \subset \PP^{n+3}$, 
\item A degree $6$ hypersurface in the weighted projective space $\PP(3,1, \dots, 1)$, and 
\item A degree $(2,2)$ complete intersection in the weighted projective space $\PP(2,1, \dots, 1)$. 
\end{enumerate}
\end{enumerate}

\subsection*{Mukai fourfolds} Finally, we list the Fano fourfolds with $i_Z= 2$ and $\rho(Z) \geq 2$ \cite{IP}.
\begin{enumerate}
\item[(F1)] $\PP^1 \times V$, where $V$ is a del Pezzo threefold,
\item[(F2)] Certain complete intersections in products of homogeneous varieties: 
\begin{enumerate}
\item A divisor of type $(1,2)$ in $\PP^2 \times \PP^3$, 
\item A divisor of type $(1,1)$ in $\PP^2 \times Q^3$, and 
\item An intersection of two divisors of type $(1,1)$ in $\PP^3 \times \PP^3$,
\end{enumerate}
\item[(F3)] The blowups $\Bl_l Q^4$, $\Bl_q Q^4$ of a quadric fourfold along a line $l$ and a conic $q$ whose plane is not contained in $Q^4$,
\item[(F4)] A double cover of $\PP^2 \times \PP^2$ branched along a divisor of type $(2,2)$,
\item[(F5)] The blowup of a cone over a quadric threefold at the vertex, 
\item[(F6)] $\PP (\mathcal{E})$, where $\mathcal{E} = \OO_{\PP^3}(-1) \oplus \OO_{\PP^3}(1)$ or $\mathcal{E}$ is the null-correlation bundle on $\PP^3$,
\item[(F7)] $\PP^1 \times \PP^3$.
\end{enumerate}

\subsection*{The linear system $|H|$} Let $Z$ be a Mukai manifold of dimension $n \geq 4$. In each case, we can write $-K_Z = (n-2)H$ for an ample divisor $H$. We now analyze the linear system $|H|$. 

\subsection*{Homogeneous $Z$} Let $Z$ be a homogeneous Mukai manifold or a linear section of one. If $\rho(Z)=1$ (case M2), then $H$ is the ample generator of $\Pic(Z)$ and the class of the Schubert divisor.  The complete linear system $|H|$ defines the Pl\"{u}cker embedding of $Z$. If  $\rho(Z)>1$, $$\mbox{i.e.,} \ \ \PP^3\times \PP^3, \ \ \PP^2 \times Q^3,\ \  F(1,3;4), \ \ \PP^1 \times \PP^1 \times \PP^1 \times \PP^1, \ \ \PP^1 \times F(1,2,3), \ \ \PP^1 \times \PP^3,$$  then $H$ is the sum of the classes of the Schubert divisors with coefficient one and   defines the Segre embedding on the product of the Pl\"{u}cker embeddings. In particular, when $Z$ is homogeneous, $H$ is very ample. 

Similarly, when $Z$ is a quadric section of the cone over $G(2,5)$ (case M1), $H$ is the class of a hyperplane section. In particular, $H$ is very ample.

\subsection*{Complete intersections} Let $Z$ be a Mukai manifold which is a complete intersection in a homogeneous variety $V$. When $V$ is projective space (cases M3 a,b,c), then $H$ is the restriction of the hyperplane class to $Z$. 
When $Z$ is a divisor of type $(1,2)$ in $\PP^2 \times \PP^3$,  a divisor of type $(1,1)$ in $\PP^2 \times Q^3$ or a complete intersection of two divisors of type $(1,1)$ in $\PP^3 \times \PP^3$ (cases F2), then $H$ is the class of  $\OO_Z(1,1)$. Hence, $H$ is very ample.

Similarly, when $Z = \PP^1 \times V$ with $V$ a Fano threefold of index 2 and degree at least 3, then $Z$ is the product of $\PP^1$ with a complete intersection in projective space. In this case, $H$ is the sum of the pullbacks of the hyperplane classes and is very ample.

\subsection*{Blowups} Let $\phi: Z = \Bl_l(\PP^5)\rightarrow \PP^5$ be the blowup of $\PP^5$ along a line $l$.  Let $L$ denote the class of  $\phi^* \OO_{\PP^5}(1)$ and let $E$ denote the class of the exceptional divisor over  $l$. Then $-K_Z = 6L - 3 E$ and $H$ is $2L - E$. Similarly, let $Z$ be one of $\Bl_l(Q_4)$, $\Bl_q(Q_4)$ as in case F3 or the blowup of the cone over $Q^3$ along the vertex (case F5). Let $L$ denote the pullback of the hyperplane class on the quadric and let $E$ denote the exceptional divisor. Then $-K_Z = 4L - 2E$ and $H = 2L - E$.

In all four cases, $H$ is very ample. It is clear that $H$ separates points and tangents away from the exceptional divisor. On the exceptional divisor $E$, we have the exact sequence 
$$0 \rightarrow \OO_Z(L-E) \rightarrow \OO_Z(2L-E) \rightarrow \OO_E(2L-E) \rightarrow 0.$$ Using that $H^1(Z, \OO_Z(L-E))=0$ by the Kodaira vanishing theorem and that $\OO_E(2L-E)$ is very ample on $E$, one  easily concludes that $H$ is very ample on $Z$.

\subsection*{Projective bundles}
When  $Z = \PP(\OO_{\PP^3}(-1) \oplus \OO_{\PP^3}(1))$,  $Z$ is the blowup of the vertex on the cone over the two-uple Veronese embedding of $\PP^3$. Then $-K_Z = 4L - 2E$ and $H=2L-E$, where $L$ is the pullback of the hyperplane class on $\PP^3$ and $E$ is the exceptional divisor over the vertex. As in the previous case, $H$ is very ample.

Let  $Z= \PP (\mathcal{E})$, where $\mathcal{E}$ is the null-correlation bundle on $\PP^3$ with $c_1(\mathcal{E}) =0$ and $c_2(\mathcal{E}) =1$. Recall that $\mathcal{E}$ may be defined via the sequence \cite{SW}
$$0 \rightarrow \OO_{\PP^3} \rightarrow \Omega \PP^3 (2) \rightarrow \mathcal{E}(1) \rightarrow 0.$$
Then, $-K_Z = 4L + 2 \xi_{\mathcal{E}}$, where $L$ is the pullback of the hyperplane class from $\PP^3$ and $\xi$ is the class of $\OO_{\PP\mathcal{E}}(1)$. The variety $Z$ can also be realized as the inverse image of a hyperplane section under the projection of $F(1,2;4)$ to the Grassmannian $G(2,4)$ (see \cite{SW}). In this description, the linear system $H$ is the restriction of the sum of the Schubert divisors on $F(1,2;4)$, hence is very ample.

\subsection*{Branched covers} When $Z$ is the double cover of $\PP^n$ branched along a sextic hypersurface, or a double cover of a quadric branched along the intersection with another quadric (cases M3 d,e), $H$ is the pullback of the hyperplane class.  Similarly, when $Z$ is the double cover of $\PP^2 \times \PP^2$ branched along a divisor of type $(2,2)$ (case F4), then $H$ is the pullback of a divisor of type $(1,1)$. In these cases, $H$ is ample and basepoint-free but not very ample. It defines the two-to-one cover. 

When $Z$ is $\PP^1 \times V$, where $V$ is a del Pezzo threefold of degree $2$ or $1$, $H$ is similarly ample but not very ample and defines the branched cover.

\subsection*{The basic analysis of our examples}
Our basic examples will be the blowups of Mukai manifolds along a general $(n-2)$-dimensional linear subsystem of $|H|$. We first observe the following. 
Let $Z$ be a Mukai manifold of dimension $n\geq 4$ and write $-K_Z = (n-2) H$. We have seen that $|H|$ is a basepoint-free linear system. In fact, except in the case of branched covers, $|H|$ is very ample. Let $V \subset |H|$ be an $(n-2)$-dimensional general linear subsystem.  

\begin{proposition}
Let $Z$ be a Fano manifold of dimension $n\geq 4$ and index $n-2$. Let $-K_Z=(n-2)H$. Then a general $(n-2)$-dimensional linear subsystem $V \subset |H|$ has base locus $Bs(V)$ a smooth curve in $Z$.  
\end{proposition}
\begin{proof}
Except when $Z$ is a branched cover  (cases M3 d,e, F1 where $V$ is a del Pezzo manifold of degree 1 or 2, F4), the linear system $|H|$ is very ample. Hence, by Bertini's Theorem, the intersection of $n-1$ general members of the linear system $|H|$ is a smooth, connected curve. In particular, it is irreducible. When $Z$ is a branched cover, $H$ is not very ample but the pullback of a very ample $L$ via the branched cover.  By Bertini's Theorem, the intersections of $n-1$ general hyperplane sections in $L$ is a smooth, irreducible curve $B$. The intersection of $n-1$ general members of $|H|$ is the branched cover of  $B$. Since $B$ is transverse to the branch locus of the cover, we conclude that the base locus of $V$ is a smooth, irreducible curve $C$.
\end{proof}

Let  $\pi: X \arrow Z$ be the blowup of $Z$ along the base curve $C$. Then we get a fibration $f: X \arrow \PP^{n-2}$ given by sections of the bundle $\pi^*H-F$. The fibers of $f$ are proper transforms on $X$ of complete intersections of $n-2$ members of $V$, and are isomorphic to these complete intersections. By the adjunction formula (applied on $Z$), the general fiber is a $K3$ surface. We also have some information about the degenerate fibers.
\begin{proposition} \label{prop-irred}
Let $Z$, $H$ and $V$ be as above.  Then the fibration $f: X \arrow \PP^{n-2}$ has irreducible fibers.
\end{proposition}
\begin{proof}
Any fiber of $f$ is a proper transform of a surface in $Z$ of the form $S=H_1 \cap \cdots \cap H_{n-2}$ where the $H_i$ are divisors in $V$. The fiber is irreducible if and only if $S$ is irreducible. Since  $C=S \cap H_0$ for another divisor $H_0 \in V$ and $H$ is ample, the restriction of $H$ to $S$ is also ample. If $S$ is reducible, $C$ is an ample Cartier divisor on a reducible surface, hence reducible. This contradicts the hypothesis on $C$. 
\end{proof}

\section{The nef cone}\label{sect-nef}

Let $Z$ be a Fano manifold of dimension $n \geq 4$ and index $n-2$. Write $-K_Z = (n-2) H$ and let $V \subset |H|$ be a general linear system of dimension $n-2$ with base locus a smooth curve $C$. Let $X = \Bl_C(Z)$. In this section, we show that the nef cone of $X$ is rational polyhedral and conclude the Morrison--Kawamata Conjecture for $\Nef{X}$.

\begin{theorem}\label{thm-nef}
The effective nef cone $\Nef{X}$ coincides with $\nef{X}$ and is a rational polyhedral cone. In particular, the Morrison--Kawamata Conjecture holds for $\Nef{X}$.
\end{theorem}

\begin{proof}
If the Fano manifold $Z$ has $\rho(Z) =1$, then $\rho(X) =2$. Let $L$ denote the pullback of the ample generator of $Z$ and let $F$ denote the exceptional divisor lying over $C$. Then $\Nef{X} =\nef{X} = \langle L, L-F\rangle$. The linear system $|L|$ defines the blowdown morphism and the linear system $|L-F|$ defines the fibration induced by $V$. Both are basepoint-free, hence nef and effective. On the other hand, the maps they define have positive dimensional fibers, hence are the two extremal rays of $\nef{X}$. We conclude that $\Nef{X}$ is a rational polyhedral cone.

We may now assume that $\rho(Z) > 1$. We will compute $\Nef{X}$ using the following lemma.

\begin{lemma}\label{lem-clean}
Let $Z$ be a smooth, projective variety of dimension $n$.  Let $N^1(Z)$ have a basis given by basepoint-free divisor classes $L_1, \dots, L_{\rho}$. Assume there exists curves $l_i$, $1\leq i \leq \rho$,  whose deformations cover $Z$ or a divisor in $Z$ and such that  $l_i \cdot L_j =\delta_{i,j}$. Let $V \subset |L_1 + \cdots + L_{\rho}|$ be a linear system of dimension $n-2$ whose base locus is a smooth curve $C$ and let $\pi: X = \Bl_C Z$ be the blowup of $C$ with exceptional divisor $F$.  Then,
\begin{enumerate}
\item $\Nef{Z} = \langle L_1, \dots, L_{\rho} \rangle$,
\item $\Nef{X} = \langle L_1, \dots, L_{\rho}, L_1 + \cdots + L_{\rho} - F \rangle,$
\end{enumerate}
  where we denote the pullback of a divisor to $X$ with the same symbol.
\end{lemma}

\begin{proof}
Since $L_1, \dots, L_{\rho}$ are basepoint-free, they are nef. Hence, the cone they generate is contained in $\Nef{Z}$. Conversely, let $D$ be a nef divisor. Since $L_1, \dots, L_{\rho}$ is a basis of $N^1(X)$, express $D= \sum_{i=1}^{\rho} a_i L_i$. Intersect $D$ with $l_i$ to obtain $l_i \cdot D = a_i \geq 0$. Hence, $\Nef{Z}$ is contained in the cone generated by $L_1, \dots, L_{\rho}$.

The divisor $L_1 + \cdots + L_{\rho}$ is ample on $Z$ since it is contained in the interior of $\Nef{Z}$. Hence, the base curve $C$ intersects every effective divisor positively. Since $l_i \cdot (L_1 + \cdots + L_{\rho}) = 1$ for every $1 \leq i \leq \rho$ and deformations of $l_i$ cover at least a divisor, we conclude that we can find representatives of the curves $l_i$ that  intersect $C$ at one point. For each $1 \leq i \leq \rho$, let $b_i$ be the proper transform of such a curve in $X$. Let $f$ be a line in $F$ lying over a point of $C$. 

Being pullbacks of nef divisors under a birational morphism, $L_i$ are nef on $X$. Similarly, $G := L_1 + \cdots + L_{\rho} - F$ is by definition basepoint-free, hence nef. We conclude that the cone generated by $L_1, \dots, L_{\rho}$ and $G$ is contained in $\Nef{X}$.  Conversely, let $D$ be a nef divisor on $X$. We can express $D =a_0 G +  \sum_{i=1}^{\rho} a_i L_i.$ Intersecting $D$ with $b_i$, we see $b_i \cdot D = a_i \geq 0$. Intersecting $D$ with $f$, we obtain $f \cdot D = a_0 \geq 0$. Hence, $D$ is contained in the cone spanned by $L_1, \dots, L_{\rho}$ and $G$. This concludes the proof.  
\end{proof}

In view of Lemma \ref{lem-clean}, to conclude the proof of Theorem \ref{thm-nef}, it suffices to define basepoint-free divisors  on $Z$ that form a basis of $N^1(Z)$ and dual curves that sweep out at least a divisor in $Z$. We now list the required divisors and curves in each case.

If $\dim(Z) \geq 5$, then define $L_1$, $L_2$, $l_1$ and $l_2$ as follows.

\noindent $\bullet \ Z = \PP^3 \times \PP^3$. Then let $L_i$ be the pullback of the hyperplane class on $\PP^3$ via the two projections and let $l_i$ be lines in the fiber of the projection $\pi_j$ for $j \not= i$ .
\smallskip

\noindent $\bullet \ Z= \PP^2 \times Q^3$. Then let $L_1 = [\pi_1^*\OO_{\PP^2}(1)]$, $L_2 = [\pi_2^* \OO_{Q^3}(1)]$ and let $l_i$ be lines in the fibers of the projections $\pi_j$ for $j \not= i$.
\smallskip

\noindent $\bullet \ Z= F(1,3;4)$. Then let $L_1$ and $L_2$ be the two Schubert divisors and let $l_1$ and $l_2$ be the dual Schubert curves.
\smallskip

\noindent $\bullet \ Z= \Bl_l (\PP^5)$.  Then let $L_1$ be the pullback of the hyperplane class on $\PP^5$  and let $L_2 = L_1 -E$, where $E$ is the exceptional divisor. Let $l_1$ be the proper transform of a line in $\PP^5$ that intersects $l$ and let $l_2$ be a line in $E$.
\smallskip

It is clear that $L_i$ and $l_i$ satisfy the assumptions of Lemma \ref{lem-clean}. We conclude that 
$$\Nef{X} = \langle L_1, L_2, L_1 + L_2 -F \rangle.$$

If $\dim(Z)=4$ and  $\rho(Z)=2$, then define $L_i$ and $l_i$ as follows.

\noindent $\bullet \ Z= \PP^1 \times V$, where $V$ is a del Pezzo threefold with $\rho(V)=1$. Then let $L_1$ and $L_2$ denote the pullback of the ample generators of $\PP^1$ and $V$ via the two projections. Let $l_1$ be a fiber of the second projection. Let $l_2$ be lines in $V$ in the fiber of the first projection.   
\smallskip

\noindent $\bullet \ Z$ is a divisor of type $(1,2)$ in $\PP^2 \times \PP^3$, a divisor of type $(1,1)$ in $\PP^2 \times Q^3$ or an intersection of two divisors of type $(1,1)$ in $\PP^3 \times \PP^3$. Then let $L_i$ be the pullbacks of the hyperplane classes via the two projections. Let $l_i$ be a line in the fiber of the projection $\pi_j$ for $j \not= i$. 
\smallskip

\noindent $\bullet \ Z$ is a blowup of a quadric as in cases F3 or F5. Then let $L_1$ be the pullback of the hyperplane class on the quadric and let $L_2 = L_1 - E$, where $E$ is the exceptional divisor. Let $l_1$ be a line on the quadric intersecting the center of the blowup and let $l_2$ be a line in $E$ lying over a point of the center of the blowup.
\smallskip

\noindent $\bullet \ Z$ is the double cover of $\PP^2 \times \PP^2$ branched along a divisor of type $(2,2)$. Let $L_i$ be the pullbacks of the hyperplane class in $\PP^2$ to $Z$ via the two projections. Let $m_i$ be a line in a fiber of the projection $\pi_j: \PP^2 \times \PP^2$, $j\not= i$. Then if $m_i$ is tangent to the branch divisor, its double cover in $Z$ is reducible. Let $l_i$ be one of the irreducible components lying over $m_i$. Since through every point of $\PP^2$, there is a line tangent to a conic, these curves cover $Z$.
\smallskip

\noindent $\bullet \ Z= \PP(\OO_{\PP^3}(-1) \oplus \OO_{\PP^3}(1))$. Then we may think of $Z$ as the blowup of the vertex of the cone over the two-uple Veronese embedding of $\PP^3$. Let $L_1$ be the pullback of the hyperplane class  and let $L_2$ be $L_1 - E$, where $E$ is the exceptional divisor. Let $l_1$ be the proper transform of a line through the cone point and let $l_2$ be a line in $E$.
\smallskip

\noindent $\bullet \ Z= \PP (\mathcal{E})$, where $\mathcal{E}$ is a null-correlation bundle. Then we may view $Z$ as a divisor in $F(1,2;4)$. Let $L_1$ and $L_2$ be the restriction of the two Schubert divisors to $Z$ and let $l_1$ and $l_2$ be the dual Schubert curves in $Z$.
\smallskip

\noindent $\bullet \ Z= \PP^1 \times \PP^3$. Then let $Li$ denote the pullback of the two hyperplane classes via the two projections and let $l_i$ be lines in the fibers of the projection $\pi_j$ for $i \not= j$. 
\smallskip

By Lemma \ref{lem-clean},  $\Nef{X} = \langle L_1, L_2, L_1 + L_2 -F \rangle.$

If $\rho(Z)=3$,  define $L_1, L_2, L_3$ and $l_1, l_2$ and $l_3$ as follows.

\noindent $\bullet \ Z = \PP^1 \times F(1,2;3)$. Then let $L_1$ be the pullback of the hyperplane class via the first projection and let $L_2, L_3$ be the pullback of the two Schubert divisors via the second projection.  Let $l_1$ be a fiber of the second projection and let $l_2, l_3$ be dual Schubert curves in the fiber of the first projection.
\smallskip

\noindent $\bullet \ Z = \PP^1 \times \Bl_p(\PP^3)$. Then let $L_1$ be the pullback of the hyperplane class on $\PP^1$, let $L_2$ be the pullback of the hyperplane class $\Lambda$ on $\PP^3$ and let $L_3$ be the pullback of the class $\Lambda - E$ from $\Bl_p (\PP^3)$, where $E$ is the exceptional divisor of the blowup. Let $l_1$ be a fiber of the second projection. Let $l_2$ be the proper transform of a line in $\PP^3$ passing through $p$ and let $l_3$ be a line in $E$.
\smallskip

By Lemma \ref{lem-clean},
$\Nef{X} = \langle L_1, L_2, L_3, L_1 + L_2 + L_3 - F \rangle.$

Finally, if $Z= \PP^1 \times \PP^1 \times \PP^1 \times \PP^1$, let $L_i$, $1\leq i \leq 4$, denote the pullbacks of $\OO_{\PP^1}(1)$ via the four projections and let $l_i$ be the inverse image of a point under the projection that contracts the $i$th factor. Lemma \ref{lem-clean} implies that 
$$\Nef{X} = \langle L_1, L_2, L_3, L_4, L_1 + L_2 + L_3 + L_4 - F \rangle.$$ 

In each case, we conclude that $\Nef{X}$ is a rational, polyhedral (in fact, simplicial) cone spanned by basepoint-free divisors. Consequently, the Morrison--Kawamata Conjecture holds for $\Nef{X}$. 
\end{proof}

If $\rho(Z)=1$, then $\Nef{X} = \langle L, L-F \rangle$. Since these bundles define a divisorial contraction and a fiber space structure, respectively, they must also be extremal rays of the movable cone. We conclude the following.
\begin{corollary}\label{remark}
If $\rho(Z)=1$, then $\Nef{X}=\Mov{X}$ and the Morrison--Kawamata conjecture holds for $\Mov{X}$. 
\end{corollary}
In fact, the cone conjecture holds for \emph{any} klt Calabi--Yau pair $(X,\Delta)$ of Picard rank 2 with $\Delta \neq 0$. See \cite{PSnote} for the proof.

\section{Generalities on cones and fibrations}\label{sect-generalities}
In this section, we record some general results that simplify the problem of finding fundamental domains. These will be applied in \S \ref{sect-infinite} to prove some infinite cases of the movable cone conjecture.

 In \S \ref{subs-looijenga} we state a theorem of Looijenga that will simplify studying the boundary of our cones. In \S \ref{subs-lifting} we study the relationship between  the movable cones of our varieties and the nef cones of the generic fibers in our $K3$ fibrations. We then give in Theorem \ref{theorem-lifting} a sufficient condition for the movable cone conjecture to hold that, crucially, only involves knowing part of the movable cones ``close to" the nef cone of $X$. This condition will be the main tool we use to verify the conjecture in the next section. 
\subsection{Looijenga's theorem} \label{subs-looijenga}
In our examples, $\Mov{X}$ is typically not polyhedral and it is difficult to construct an explicit fundamental domain for the action of birational automorphisms. The following powerful general theorem of Looijenga allows us to conclude the existence of a fundamental domain given seemingly weaker information, namely a polyhedral cone whose translates cover the interior of our cone.

\begin{theorem}[\cite{Looijenga}, Proposition-Definition 4.1, Application 4.15] \label{thm-looijenga}
Let $V$ be a real vector space with rational structure, $C$ an open
strictly convex cone in $V$, and $\Gamma$ a subgroup of $GL(V)$ that
stabilizes the cone $C$ and a lattice in $V(\QQ)$. Let $C_+$ denote
the convex hull of all the rational points in the closed cone $\overline{C}$. Then the following
are equivalent:
\begin{enumerate}
\item[(i)] There exists a rational polyhedral cone $\Pi$ in $C_+$ with $C_+ = \Gamma \cdot \Pi$.
\item[(ii)] There exists a rational polyhedral cone $\Pi$ in $C_+$ with $C \subset \Gamma \cdot \Pi$.
\end{enumerate}
Moreover, if these conditions are satisfied, then there exists a
rational polyhedral fundamental domain for the action of $\Gamma$ on
$C_+$.
\end{theorem} 

We will apply this theorem with $V=N^1(X)$ with its natural rational
structure, $C$ the open cone $\operatorname{Int}(\Mov{X})$, and $\Pi$
the image of $\PsAut(X)$ inside $GL(N^1(X))$. Note that $\Mov{X}
\subset \Mov{X}_+$. If we can find a rational polyhedral
cone $\Pi \subset \Mov{X}$ such that $\Gamma \cdot \Pi$ covers $C$,
Theorem \ref{thm-looijenga} guarantees that $\Gamma \cdot \Pi = C_+$, hence that
$\Mov{X}= \Mov{X}_+$. Moreover, $\Gamma$ acts on $\Mov{X}$
with a rational polyhedral fundamental domain, as required by the cone
conjecture. Hence, it is enough to find a $\Pi$ whose translates cover the
interior of $\Mov{X}$. This simplifies the problem in two ways.
First, we need not find a precise fundamental domain, and
second, we need not worry about the boundary of $\Mov{X}$.

For convenience we introduce the following (nonstandard) terminology.
\begin{definition}
In the notation of Theorem \ref{thm-looijenga}, a rational polyhedral cone $\Pi$ such that $C \subseteq \Gamma \cdot \Pi$ will be called a \emph{covering domain} for the action of $\Gamma$ on the cone $C$. 
\end{definition}
By Theorem \ref{thm-looijenga} and the paragraph above, then, to prove the movable cone conjecture for $X$, it will be enough to find a covering domain for the action of $\PsAut(X)$ on $\operatorname{Int} ( \Mov{X} )$.

\subsection{Restriction and lifting} \label{subs-lifting}
In this section we relate properties of line bundles and automorphisms on the generic fiber of a fibration to those on the total space. This will allow us to ``lift" information on the generic fibers of our examples --- which are $K3$ surfaces, and hence well-understood --- to the total space, whose geometry may be more complicated.

We start by introducing relative versions of the notions of \S \ref{sect-intro}. As usual, let $X$ be a $\QQ$-factorial projective variety. If $f : X  \arrow S$ is a contraction (meaning a surjective morphism with connected fibers), denote by $N^1(X/S)$ the vector space $\left ( \Pic(X)/\equiv_S) \right) \otimes \RR$. Here $\equiv_S$ denotes numerical equivalence over $S$: for line bundles $L_1$ and $L_2$ on $X$, we define $L_1 \equiv_S L_2$ to mean that $L_1 \cdot C = L_2 \cdot C$ for all curve $C \subset X$ contracted by $f$. There is a quotient map $q: N^1(X) \arrow N^1(X/S)$. 

A line bundle $L$ is called $f$-effective, respectively $f$-nef or $f$-movable, if $f_*L \neq 0$, respectively $L \cdot C \geq 0$ for all curves $C$ contracted by $f$, or $$\operatorname{codim} (\operatorname{Supp} (\operatorname{coker} (f^* f_* O_X(L) \arrow O_X(L)))) \geq 2.$$ As in the absolute case, we define the relative nef cone $\nef{X/S}$ and relative movable cone $\mov{X/S}$ and their effective subcones $\Nef{X/S}$ and $\Mov{X/S}$. 

Finally, an SQM $\alpha: X \dashrightarrow X'$  of $X$ over $S$ means another $\QQ$-factorial variety $X'$ with a contraction $g: X' \arrow S$ such that $f = g \circ \alpha$ as rational maps. We define $\PsAut(X/S)$  to be the group of all SQMs over $S$ from $X$ to itself.

\begin{proposition} \label{prop-restriction}
Let $X$ be a smooth variety with $b_1(X)=0$. Let $f: X \arrow \PP^k$
be an equidimensional fibration of relative dimension at least 2, with
irreducible fibers in codimension 1, and $b_1(X_\eta)=0$. Then
\begin{enumerate}
\item[(i)] Restriction of divisors to the generic fiber $X_\eta$ induces an isomorphism $N^1(X/\PP^k) \iso N^1(X_\eta)$.
\item[(ii)] The isomorphism in $(i)$ identifies the relative effective movable cone $\Mov{X/\PP^k}$ with the effective movable cone $\Mov{X_\eta}$.
\item[(iii)] In particular, if $f$ has relative dimension 2, then $\Mov{X/\PP^k} \iso \Nef{X_\eta}$.
\end{enumerate}
In particular, these statements hold for our examples of blowups of Fano manifolds of index $n-2$.
\end{proposition}
\begin{proof}
Restriction of divisors gives a surjective homomorphism $\rho: \Pic(X) \arrow \Pic(X_\eta)$. The kernel $\operatorname{ker} \rho$ contains $\operatorname {ker} \left ( \Pic(X) \arrow \Pic(X/\PP^k) \right)$, so there is an induced surjective homomorphism $\overline{\rho}: \Pic(X/\PP^k) \arrow  \Pic(X_\eta)$. The kernel of $\rho$ is spanned by irreducible divisors on $X$ which do not map onto $\PP^k$, but by our hypothesis any such divisor is a multiple of $f^* O(1)$. Since this line bundle is evidently trivial on any curve contracted by $f$, this shows that $\overline{\rho}$ is an isomorphism. Finally, tensoring by $\RR$, we get the required isomorphism $N^1(X/\PP^k) \iso N^1(X_\eta)$. 

To prove (ii), first we will show that the effective cones are identified. For a line bundle $D$ on $X$, let $D_\eta$ denote its restriction to $X_\eta$. By definition, $D$ is an $f$-effective divisor if there exists an open subset $U \subseteq \PP^k$ such that $D(f^{-1}(U))$ is nonzero. A nonzero element of this group restricts to give a nonzero section of $D_\eta$. Conversely, if $H^0(X_\eta,D_\eta) \neq 0$, then Grauert's theorem \cite[Corollary III.12.9]{Hartshorne} shows that $D(f^{-1}(U)) \neq 0$ for some open set $U \subseteq \PP^k$. 

Next suppose that $D$ is a line bundle such that $D_\eta$ is movable. Choose divisors $\Delta_1$, $\Delta_2$ of $D_{|X_\eta}$ such that $\cap_i \Delta_i$ has codimension 2 in $X_\eta$. Again by Grauert's theorem, these divisors extend over an open set $U$ in the base $\PP^k$, and since they intersect only in codimension 2 on the general fiber, we can choose $U$ such that their intersection remains of codimension 2 in $f^{-1}(U)$. Taking the Zariski closure of each $\Delta_i$, we get divisors $\overline{\Delta_i } \in | m_i(f^* O(1)) + D |$ for some integers $m_i$ such that $\cap_i \overline{\Delta_i}$ has codimension 2 in $X$. Now multiplying up by appropriately chosen sections of the bundles $(m-m_i) (f^* O(1) )$ we get sections of a single bundle $D+m(f^* O(1) )$ whose common zero locus has codimension 2 in $X$. This shows that the class of $D$ in $N^1(X/\PP^k)$ belongs to the cone $\operatorname{Mov}^e(X/\PP^k)$. 

Conversely, suppose $D$ is an $f$-movable line bundle. By definition, this means that $$\operatorname{Supp} \left( \operatorname{coker} \left(f^*f_* O_X(D) \rightarrow O_X(D) \right) \right)$$ has codimension at least 2 in $X$. A component of this support is either disjoint form $X_\eta$ or else intersects $X_\eta$ in a set of codimension at least 2. This shows that $\operatorname{coker} \left( H^0(X_\eta, D_\eta) \rightarrow D_\eta \right)$ is supported on a set of codimension at least 2: in other words, $D_\eta$ is movable. 

Finally, $(iii)$ follows immediately from $(ii)$ since by Zariski--Fujita, a movable divisor on a surface is semi-ample, hence nef, so the nef and movable cones coincide.

The hypotheses apply to our examples. Any such $X$ is a blowup of a Fano variety, so has $b_1(X)=0$. The associated fibration is equidimensional since each fiber is a complete intersection.  The generic fiber is a K3 surface, so has $b_1(X_\eta)=0$. Finally, by Proposition \ref{prop-irred}, all fibers of $f$ are irreducible.
\end{proof}

\begin{proposition} \label{prop-auts}
Let $f: X \arrow S$ be a contraction morphism with generic fiber $X_\eta$. Assume that $K_X \cdot C \geq 0$ for all curves $C$ contracted by $f$. Then $\PsAut(X/S) \iso \PsAut(X_\eta)$. In particular if $X_\eta$ is a surface, then $\PsAut(X/S) \iso \Aut(X_\eta)$. 
\end{proposition}
\begin{proof}
A pseudo-automorphism of $X$ over $S$ preserves the fibers of $f$, and so restricts to a pseudo-automorphism of $X_\eta$. This gives an injective homomorphism $\PsAut(X/S) \iso \PsAut(X_\eta)$.

Conversely, a pseudo-automorphism of $X_\eta$ extends uniquely to a birational map $\alpha: X \dashrightarrow X$ commuting with $f$. But now $K_X \cdot C \geq 0$ for all curves $C$ contracted by $f$: in other words, $X$ is relatively minimal over $S$. Hence by \cite[Theorem 3.52]{KollarMori} any birational map from $X$ to itself over $S$ must be a pseudo-automorphism.
\end{proof}
Putting these two propositions together yields the following.
\begin{corollary}
Let $f: X \arrow \PP^k$ be a fibration satisfying the hypotheses of Propositions \ref{prop-restriction} and \ref{prop-auts}. If the movable cone conjecture holds for the generic fiber $X_\eta$, then it also holds for the morphism $f$. That is, the group $\PsAut(X/\PP^k)$ acts on the cone $\Mov{X/\PP^k}$ with a rational polyhedral fundamental domain. 

In particular if $X_\eta$ is a $K3$ surface, then the cone conjecture holds for $f$. 
\end{corollary}

To prove the Morrison--Kawamata conjecture in our examples, we need to lift the fundamental domain from the relative cone $\Mov{X/\PP^k}$ to the absolute cone $\Mov{X}$. The following theorem, proved in an earlier work \cite{CPS}. shows how to do this. 


\begin{theorem}[\cite{CPS},  Proposition 5.8, Theorem 5.9] \label{theorem-lifting}
Let $X$ be one of our examples, and let $V$ denote a fundamental domain for the action of $\PsAut(X/\PP^k)$ on $\Mov{X/\PP^k}$. Suppose that we can find a finite collection $\{X_i\}_{i \in I}$ of SQMs of $X$ with the following properties:
\begin{enumerate}
\item[(a)] Each effective nef cone $\Nef{X_i}$ is rational polyhedral;
\item[(b)] $-K_X$ is nef on each $X_i$;
\item[(c)] Every codimension-1 face of each of the cones $\Nef{X_i}$
  which intersects the big cone is dual to the class of a $K$-trivial
  curve;
\item[(d)] The cone $V$ is covered by the union $q( U) = \bigcup_{i \in I} q(\Nef{X_i})$. (Here $q$  denotes the quotient map  $N^1(X) \arrow N^1(X/\PP^k)$.) 
\end{enumerate}
Then the movable cone conjecture holds for $X$. 
\end{theorem}

\subsection{$K3$ surfaces}
In this section we summarize the results we need from the theory of $K3$ surfaces. Recall that the $K3$ lattice $\Lambda$ is the unique even unimodular lattice of signature $(3,19)$. If $X$ is a $K3$ surface with Picard lattice$\Pic(X)$, then there is an embedding $N \hookrightarrow \Lambda$, and the complement $T:=\Pic(X)^\perp \subset \Lambda$ is called the \emph{transcendental lattice} of $X$. The \emph{Global Torelli Theorem} of Piatetskii--Shapiro and Shafarevich characterizes automorphism of the surface $X$ in terms of the algebra of these lattices. We state it in a form that is convenient for our purposes. 

\begin{theorem}[Global Torelli Theorem] \label{theorem-torelli}
Let $X$ be a $K3$ surface over an algebraically closed field of characteristic zero, and $\alpha \in O(\Pic(X))$ an automorphism of its N\'eron--Severi lattice. Then $\alpha$ is induced by an automorphism $f: X \arrow X$ if and only if:
\begin{enumerate}
\item $\alpha$ preserves the set of nodal classes in $\Pic(X)$.
\item $\alpha$  can be extended to an automorphism $\tilde{\alpha}: \Lambda \arrow \Lambda$ such that the induced automorphisms on the (naturally isomorphic) discriminant groups $\Pic(X)^*/\Pic(X)$ and $T^*/T$ are the same. 
\end{enumerate}
In particular $\alpha \in O(\Pic(X))$ is any automorphism satisfying the first condition, then some positive power of $\alpha$ satisfies the second condition too, hence extends to an automorphism of $X$.
\end{theorem}
Here a {\em nodal class} is the class of a smooth rational curve on $X$. Any such class $x$ satisfies $x^2=-2$, by adjunction. Any class $x \in \Pic(X)$ with $x^2=-2$ is called a {\em $(-2)$-class}.

This has a particularly useful consequence for $K3$ surfaces of Picard number 2.
\begin{theorem}[\cite{GLP} Example 3.5] \label{corollary-torelli2}
Let $X$ be a $K3$ surface with $\rho(X)=2$ over an algebraically closed field of characteristic zero. Suppose that the order of the discriminant group $\Pic(X)^*/\Pic(X)$ does not equal 2,3,4,5,8,11, or 25. Then an automorphism $\alpha \in O(\Pic(X))$ preserving the nodal classes is induced by an automorphism $f: X \arrow X$ if and only if $\alpha$ acts as $\pm \operatorname{id}$ on $\Pic(X)^*/\Pic(X)$. 
\end{theorem}
We will use these theorems to produce automorphisms of the $K3$ surfaces appearing in our examples. As stated they only apply to surfaces over algebraically closed fields, but Kawamata \cite[Remark 2.2]{Kawamata1997} showed how to extend them to any base field. We use a variant with stronger hypotheses, but that suffices for our purposes.
\begin{corollary} \label{corollary-torellispecial}
Let $X$ be a $K3$ surface over a field $k$ of characteristic zero, and let $\overline{X}=X \times_k \overline{k}$ be the base-change to the algebraic closure of $k$. Suppose that $N^1(X)=N^1(\overline{X})$. Then the conclusions of Theorems \ref{theorem-torelli} and \ref{corollary-torelli2} apply to $X$. 

In particular, if there are no $(-2)$-classes in $\Pic(X)$, then if $O(\Pic(X))$ is infinite, then so too is $\Aut^*(X)$.
\end{corollary}
\begin{proof}
The proof is the same as \cite[Remark 2.2]{Kawamata1997} except that we might need to extend $\alpha \in O(\Pic(X))$ by $-\operatorname{Id}$ on $T$. Any such extension still commutes with the action of the Galois group, and so it comes from an automorphism of $X$. 

For the final statement, if there are no $(-2)$-classes in $\Pic(X)$ then the set of nodal classes is empty. Since $\Pic(X)^*/\Pic(X)$ is a finite group, so is its automorphism group. So $\Aut^*(X)$ contains the finite-index subgroup $\operatorname{ker} \left( O(\Pic(X)) \arrow \operatorname{Aut}(\Pic(X)^*/\Pic(X)) \right)$. 
\end{proof}

\subsection{Flops}
In our examples, the SQMs we need to verify the conditions of Theorem \ref{theorem-lifting} will arise as flops of our original varieties. In this section we collect the facts we need about flops. We start with some terminology.
\begin{definition}
Let $\varphi: X \dashrightarrow X'$ be an SQM. Let $U$ and $U'$ be the maximal open sets in $X$ and $X'$ respectively on which $\varphi$ and $\varphi^{-1}$ are isomorphisms. Then we call $X \setminus U$ the \emph{centre} of the SQM $\varphi$, and $X' \setminus U'$ the \emph{concentre}.
\end{definition}

\begin{proposition}
Let $\varphi: X \dashrightarrow X'$ be an SQM. Let $\varphi^*: N^1(X) \iso N^1(X')$ be the isomorphism given by proper transform of divisors, and let $\varphi_*: N_1(X) \iso N_1(X')$ be the dual isomorphism. Then if $C$ is a curve on $X$ which is disjoint from the centre of $\varphi$, we have $\varphi_*[C]=[\varphi(C)]$ in $N_1(X')$.
\end{proposition}
\begin{proof}
Since $C$ is disjoint from the centre of $\varphi$, there is an open set $U$ containing $C$ on which $\varphi$ is an isomorphism.
\end{proof}

For the next statement, let $N_1(X)$ denote the vector space generated by $1$-cycles on $X$ modulo numerical equivalence. Then $\overline{\Curv(X)}$ denotes the {\em cone of curves} of $X$, meaning the closed cone in $X$ generated by classes of effective 1-cycles (equivalently, irreducible curves). (By definition of nefness, $\overline{\Curv(X)}$ is dual to the nef cone $\nef{X}$.) 
\begin{proposition}
Let $(X,\Delta)$ be a klt Calabi--Yau pair. Suppose that $R \subset \overline{\Curv(X)}$ is an extremal ray, and suppose there is a movable effective divisor $D$ on $X$ such that $D \cdot R <0$. Then the log flip of $R$ exists. In particular if $R \subset K_X^\perp$, then the flop of $R$ exists. 
\end{proposition}
\begin{proof}
Since $(X,\Delta)$ is klt and $D$ is effective, for any sufficiently small $\epsilon >0$ the pair $(X,\Delta+\epsilon D)$ is also klt. Moreover since $K_X + \Delta \equiv 0$ we also have $(K_X + \Delta + \epsilon D) \cdot R = D \cdot R <0$. So $R$ is a $(K_X+\Delta+\epsilon D)$-negative extremal ray. By the Cone Theorem \cite[Theorem 3.7]{KollarMori}, the contraction $f:X \arrow Z$ of $R$ exists. Any curve with class in $R$ must lie in $\operatorname{Bs}(D)$, which has codimension at least 2 since $D$ is movable, so the contraction $f$ is small. Hence its log flip exists by Birkar--Cascini--Hacon--McKernan \cite[Corollary 1.4.1]{BCHM}.
\end{proof}

\subsection{Flips}
In this subsection we show that for the first of our examples  in Section \ref{sect-infinite}, flipping contractions do not occur. As a consequence, the decomposition of the movable cone into nef cones is compatible with restriction to the generic fiber. This will allow us to prove the conjecture in this case with no need for explicit analysis of the geometry.

The main input is the following important result of Kawamata \cite{Kaw2}.
\begin{theorem}[Kawamata]
Let $(X,\Delta)$ be a klt pair, and $f: X \rightarrow Z$ be a birational $(K_X+\Delta)$-negative extremal contraction. Let $E$ denote the exceptional locus of $f$ and define 
$$n := \operatorname{dim} E - \operatorname{dim} f(E).$$
Then $E$ is covered by rational curves $C$ such that $f(C)$ is a point and $-(K_X+\Delta) \cdot C < 2n$.
\end{theorem}
This immediately gives
\begin{corollary} \label{cor-length}
Let $X$ be a klt variety such that $-K_X=kH$ for a line bundle $H$ and an integer $k$. If $f: X \arrow Z$ is a $K_X$-negative extremal contraction with exceptional locus $E$ then $n=\operatorname{dim} E - \operatorname{dim} f(E)>\frac{k}{2}$. 
\end{corollary}
As the dimension of $X$ grows, this bound becomes less useful, but it will give us exactly the information we need in dimension 4:
\begin{theorem} \label{theorem-noflips}
Let $X$ be a variety satisfying the following hypotheses:
\begin{enumerate}
\item $\operatorname{dim} X = 4$;
\item there exists a line bundle $H$ such that $-K_X=2H$;
\item $X$ is obtained from a smooth variety $Y$ by a sequence of flops.
\end{enumerate}
Then $X$ has no flipping contractions.
\begin{proof}
First assume $X$ itself is smooth. Suppose $f:X \arrow Z$ is a contraction of a $K$-negative extremal ray $R$. By the Cone Theorem, $R$ is minimally spanned by the class of a rational curve $C$. Standard estimates from deformation theory \cite[Theorem II.1.14]{Kollar} show that the space of deformations of $C$ inside $X$ is at least $-K_X \cdot C + \operatorname{dim} X - 3 \geq 2 +4-3=3$. If those deformations swept out a locus of dimension at most 2 in $X$, then by Bend and Break \cite[Corollary II.5.6]{Kollar} some deformation of $C$ must be reducible, contradicting the fact that $R$ is extremal and the class of $C$ is minimal. So deformations of $C$ must sweep out a locus of dimension at least 3, and therefore $f$ is not a small contraction.

In general $X$ will be mildly singular, but there is a sequence of flops $Y \dashrightarrow Y_1 \dashrightarrow \cdots \dashrightarrow Y_n =X$. Suppose $f:X \arrow Z$ is a $K$-negative extremal contraction, and let $C$ be a general curve with class minimally spanning the contracted extremal ray $R$. 

The first possibility is that (the proper transform of) $C$ is not contained in the cocenter of any of the flops $Y_{i-1} \dashrightarrow Y_i$. Then it has a proper transform $C'$ on $Y$, and the generic deformation of $C'$ gives a deformation of $C$. By the previous argument, this deformation space has dimension 3, and so by Bend-and-Break (on $X$) the locus of deformations must have dimension 3, hence $f$ is not small.

The second possibility is that (the proper transform of) $C$ is contained in the cocenter of some flop, say $Y_{i-1} \dashrightarrow Y_i$. According to Corollary \ref{cor-length}, the exceptional locus $E$ of the contraction $f:X \arrow Z$  must satisfy  $n=\operatorname{dim} E - \operatorname{dim} f(E)>\frac{2}{2}=1$. So if we assume that $f$  is a small contraction, we must have $\operatorname{dim} f(E)=0$. That is, all curves $D$ in $E$ are numerically proportional to $C$ and satisfy $K_X \cdot D <0$. On the other hand, since $C$ is contained in the cocenter of the flop  $Y_{i-1} \dashrightarrow Y_i$, the proper transform of $E$  on $Y_i$ must be an irreducible component $E_i$ of the cocenter of the flop. Choosing a general flopping (i.e $K$-trivial) curve inside $E_i$, its proper transform under the map $Y_i \dashrightarrow X$ is then a curve in $E$, hence is $K$-negative. But $Y_i \dashrightarrow X$ is a sequence of flops, and flops preserve the property of $K$-triviality. This contradicts the assumption that $f$ is small. 
\end{proof}
\end{theorem}

\section{Finite cases}\label{sect-finite}
In the remaining sections, we will verify that the movable cone conjecture is true for a number of examples of Fano manifolds of index $n-2$. In this section, we focus on examples where the movable cone is a rational polyhedral cone. We will analyze $\PP^1 \times \PP^3$, the blowup of the cone over a quadric threefold at the vertex, and $\PP^1 \times V$, where $V$ is a del Pezzo threefold of Picard number 1.

\subsection{$\PP^1 \times \PP^3$}
The Chow ring of $\PP^1 \times \PP^3$ is isomorphic to $\ZZ[L_1, L_2]/ (L_1^2, L_2^4)$, where $L_1$ and $L_2$ are the pullbacks of the hyperplane classes from $\PP^1$ and $\PP^3$, respectively. Let $H= \OO_{\PP^1 \times \PP^3}(1,2)$ and let $V \subset |H|$ be a general $2$-dimensional linear system. The base locus of $V$ is a smooth curve $C$ of genus 17 with class $12 L_1 L_2^2 + 8 L_2^3$. Blowing up $C$, we obtain a variety $X$ with a fibration to $\PP^2$, $f: X \rightarrow \PP^2$, whose generic fiber is a $K3$ surface $X_{\eta}$. The Picard rank of $X_{\eta}$ is two and $\lambda_1$, $\lambda_2$, the restrictions of $L_1$ and $L_2$ can be taken as generators.

\begin{lemma}
In the basis $\lambda_1, \lambda_2$, the intersection matrix on $\Pic(X_{\eta})$ is
\begin{align*}
\begin{pmatrix}
0 & 4 \\
4 & 4
\end{pmatrix}
\end{align*}
There are no $(-2)$-curves on $X_{\eta}$.
\end{lemma}
\begin{proof}
We calculate the intersection matrix as follows
$$\lambda_1^2 = L_1^2 (L_1 + 2L_2)^2 = 0, \quad \lambda_1\lambda_2 = L_1 L_2 (L_1 + 2L_2)^2 = 4, \quad \lambda_2^2 = L_2^2 (L_1 + 2L_2)^2 = 4.$$
If $a\lambda_1 + b\lambda_2$ is a class, then its self-intersection is $4 (2 ab + b^2)$. Since this number is always divisible by $4$, it cannot represent $-2$.
\end{proof}

\begin{theorem}
$$\Mov{\Bl_C \PP^1 \times \PP^3}= \langle L_1, L_2, L_1 + 2L_2 -F, 4L_2 - F\rangle.$$ In particular, the cone conjecture holds for $\Bl_C \PP^1 \times \PP^3$.
\end{theorem}
\begin{proof}
Let $D= aL_1 + bL_2 - cF$ be a movable divisor. By abuse of notation, denote the restriction of $L_1$ and $L_2$ to a general fiber $X_t$ of $f$ also by $\lambda_1$ and $\lambda_2$. Then the classes $\lambda_1$ and $2\lambda_2 - \lambda_1$ define elliptic curves that sweep out $X_t$. Hence, $D$ must intersect nonnegatively with the classes of these elliptic curves. Since $$\lambda_1 \cdot (\lambda_1 + 2 \lambda_2) = 8, \quad (2\lambda_2 - \lambda_1) \cdot (\lambda_1 + 2\lambda_2) = 16,$$ we conclude that $$b \geq 2c, \quad 2a + b \geq 4c.$$ Considering lines $l_1, l_2, f$ in $\PP^0 \times \PP^3, \PP^1 \times \PP^0$ and $F$, we conclude that $a,b,c \geq 0$. The divisors $L_1, L_2, L_1 + 2L_2 - F$ are basepoint-free, hence movable and effective. Suppose $4L_2 - F$ is movable and effective, then we can conclude the theorem as follows.
If $a \geq c$, then $D = c(L_1 + 2L_2 -F) + (a-c)L_1 + (b-2c) L_2$. Since $a\geq c$ and $b \geq 2c$, we conclude that $D$ is in the nonnegative span of $L_1, L_2$ and $L_1 + 2L_2-F$.
If $a < c$, write $c= a + r$, with $r>0$. Then $$D = a (L_1 + 2L_2 - F) + r(4L_2 -F) + (b-2a-4r)L_2.$$ But $$2 a + b \geq 4c =4a + 4r,$$ hence $b \geq 2a+4r$ and $D$ is in the nonnegative span of $L_2, 4L_2 - F$ and $L_1 + 2L_2 -F$.
There remains to show that $4L_2 - F$ is movable. Consider the projection of the curve $C$ to $\PP^3$ and denote it by $D$. Then $D$ is cut out by 3 quartic surfaces. First, $D$ is a smooth curve of degree 12 and genus 17. Since $L_2^3 \cdot C =1$, if $C$ intersected a fiber of $\pi_2$ with multiplicity two or more, it would contain the entire fiber, contradicting the smoothness of $C$. We conclude that $D$ is smooth. Moreover, $D$ is contained in a 3-parameter family of irreducible quartic surfaces obtained by the projection of the $K3$ surfaces containing $C$. Since the $K3$ surfaces cut out $C$, the quartic surfaces cut out $D$. We conclude that the linear system $|4L_2 - F|$ is at least 2-dimensional and its base locus $\pi_2^{-1} (D)$ has codimension 2. This concludes the proof of the theorem.
\end{proof}

\subsection{$\PP^1 \times V$}
Let $V$ be a threefold with index $2$ and Picard rank 1. Recall that these are a codimension 3 linear section of $G(2,5)$ in the Pl\"{u}cker embedding, a complete intersection of 2 quadrics in $\PP^5$, a cubic 3-fold in $\PP^4$, a double cover of $\PP^3$ branched along a quartic and a sextic hypersurface in $\PP(3,2,1,1,1)$ (see \cite{CPS} and \cite{Fujita}). Let $L_1$ and $L_2$ be the hyperplane classes on $\PP^1$ and $V$, respectively. The degree of $V$ is $5 \geq d \geq 1$ in decreasing order.
Let $H$ be the class $L_1 + L_2$. Take a two-dimensional linear system $V \subset |H|$; then the base locus of $V$ is a curve $C$ of genus $2d+1$, where $d$ is the degree of the Fano 3-fold. Blowing up $C$ we obtain a $K3$ fibration $f: X \rightarrow \PP^2$ whose general fiber is a $K3$ surface $X_{\eta}$. Let $\lambda_1, \lambda_2$ denote the restriction of $L_1$ and $L_2$, respectively, to $X_{\eta}$. These generate $\Pic(X_{\eta})$.

\begin{lemma} \label{prop-fano3fold}
In terms of the basis $\lambda_1, \lambda_2$ the intersection matrix on $\Pic(X_{\eta})$ is
\begin{align*}
\begin{pmatrix}
0 & d \\
d & 2d
\end{pmatrix}
\end{align*}
When $d>1$, there are no $(-2)$-curves on $X_{\eta}$.
\end{lemma}

\begin{proof}
The intersection number can be computed as follows
$$\lambda_1^2 = L_1^2 (L_1+L_2)^2 = 0, \quad \lambda_1 \lambda_2 = L_1 L_2 (L_1 + L_2)^2 = d, \quad \lambda_2^2 = L_2^2 (L_1 + L_2)^2 = 2d.$$
The self-intersection of a class is given by $(a\lambda_1 + b \lambda_2)^2 = 2d b(a+b)$. Hence, the self-intersection of any class is divisible by $2d$. Therefore, if $d>1$, this quadratic form cannot represent a class of self-intersection $-2$.
\end{proof}

\begin{corollary}
If $d>1$, then $$\Mov{\Bl_C(\PP^1 \times V)} = \langle L_1, L_2, L_1 + L_2 - F, 2L_2 - F \rangle.$$ In particular, the cone conjecture holds for $\Bl_C(\PP^1 \times V)$.
\end {corollary}

\begin{proof}
Let $D = a L_1 + bL_2 - cF$ be a movable divisor class. By abuse of notation, denote the restriction of $L_1$ and $L_2$ to a general fiber $X_t$ of $f$ by $\lambda_1$ and $\lambda_2$. Then the classes $\lambda_1$ and $\lambda_2 - \lambda_1$ are classes of elliptic curves that sweep out the fiber $X_t$. Consequently, $D$ must intersect them nonnegatively. We conclude that $d(b-c) \geq 0$ and $d(a+b-c) \geq 0$. Moreover, $D$ must intersect the curves $l_1$, $l_2$ and $f$ nonnegatively since these curves cover $X$ or a divisor in $X$. Hence, $a,b,c \geq 0$.

The divisor classes $L_1, L_2$ and $L_1 + L_2 - F$ are basepoint-free, hence they are movable. Assuming that $2L_2 - F$ is movable and effective, we shoe that $D$ lies in the cone spanned by these 4 divisor classes. If $a \geq c$, then
$$D = c(L_1 + L_2 - F) + (a-c)L_1 + (b-c)L_2,$$ is a nonnegative linear combination of $L_1, L_2$ and $L_1 + L_2 - F$ since $b\geq c$. If $a < c$, then $$D = a(L_1 + L_2 - F) + (c-a)(2L_2 - F) + (a+b-c)L_2$$ is a nonnegative linear combination of $L_2, 2L_2 - F$ and $L_1 + L_2 - F$ since $a+b \geq c$.

To conclude the proof, there remains to show that $2L_2 - F$ is movable. Consider the projection of $C$ to $V$. We analyze the elements of the linear system $|\OO_V(2)|$ that contain $B=\pi_2(C)$. The curve $B$ is a curve of degree $3d$ and genus $2d+1$ in $\PP^{d+1}$. By Serre duality, the line bundle $\OO_B(2)$ is non-special. Hence, by Riemann-Roch, $h^0(\OO_B(2))=4d$. Since $$h^0(\OO_{\PP^{d+1}}(2)) = {d+3 \choose 2},$$ we conclude that $B$ is contained in a linear system of quadrics of dimension at least $${d+3 \choose 2} - 4d.$$ In fact, this linear system of quadrics cuts out $B$. We conclude that $2L_2 - F$ is effective and its base locus is contained in the proper transform of $\pi_2^{-1}(B)$. Hence, $2L_2 - F$ is movable and effective. This concludes the proof.
\end {proof}

\begin{proposition}
When $d=1$, $\Mov{\Bl_C (\PP^1 \times V)} = \langle L_1, L_2, L_1 + L_2 - F \rangle$ and the cone conjecture holds for $\Bl_C (\PP^1 \times V)$.
\end{proposition}

\begin{proof}
Since $L_1$, $L_2$ and $L_1 + L_2 - F$ are basepoint-free, the cone they span is in the movable cone. Conversely,
let $D = a L_1 + bL_2 - cF$ be a movable divisor on $\Bl_C (\PP^1 \times V)$.
By abuse of notation, let $\lambda_1$ and $\lambda_2$ denote the restrictions of $L_1$ and $L_2$ to a general fiber $X_t$. Then $\lambda_1$ is the class of an elliptic curve that sweeps out $X_t$. The class $\lambda_2 - 2 \lambda_1$ is the class of a $(-2)$-curve on $X_t$. The Zariski closure of these curves cannot have codimension two or more in $X$. Consequently, any movable divisor must intersect these curves nonnegatively. We conclude that $b \geq c$ and $a \geq c$. As in the previous cases, we must have $a, b, c \geq 0$. Hence, we can write $$D = c(L_1 + L_2 - F) + (a-c) L_1 + (b-c) L_2$$ as a nonnegative linear combination.
\end{proof}

\subsection{The blowup of the cone over a quadric threefold at the vertex}
Let $Z$ be the blowup of the vertex of the cone over the quadric threefold. Let $L$ denote the pullback of the hyperplane class and let $E$ be the exceptional divisor. Then the nonzero top intersection numbers are
$$L^4 = 2, \quad E^4 = -2.$$ The class of $H$ is $2 L-E$. Let $V$ be a two-dimensional linear system $V \subset |H|$. Then the base locus of $V$ is a smooth curve of genus 16. Blowing up the base locus, we get a fibration $f: X \rightarrow \PP^2$ whose generic fiber $X_{\eta}$ is a $K3$ surface. The Picard group of $X_{\eta}$ is generated by $\lambda, e$, the restrictions of $L$ and $E$, respectively.

\begin{lemma}
In terms of the basis $\lambda,e$, the intersection matrix is
\begin{align*}
\begin{pmatrix}
8 & 0 \\
0 & -2
\end{pmatrix}
\end{align*}
The only $(-2)$-class on $X_{\eta}$ is $e$.
\end{lemma}
\begin{proof}
The intersection product can be calculated by $$\lambda^2 = L^2 (2 L- E)^2 = 8, \quad e^2 = E(2L - E)^2 = -2, \quad \lambda e = L E(2 L - E)^2 = 0.$$ A ($-2$) curve satisfies the equation $(a\lambda + be)^2 = 8 a^2 - 2b^2 =-2$. Writing $(2a+b)(2a-b) =-1$, we see that the only integral solutions are $a=0, b=1$.
\end{proof}

\begin{theorem}
The effective movable cone of $X$ is given by $$\Mov{X} = \langle L, L-E, 2L-E-F, 3L-3E - F \rangle.$$ In particular, the cone conjecture holds for $X$.
\end{theorem}

\begin{proof}
Let $D= aL - b E -c F$ be a movable divisor on $X$. Then the intersection of $D$ with any curve that covers a divisor has to be nonnegative. Taking lines in the quadric, proper transforms of lines in the quadric passing through the cone point, and lines in the fibers of the two exceptional divisors $E$ and $F$, we see that $a,b,c \geq 0$ and $a \geq b$. The base curve $C$ intersects the exceptional divisor $E$ in two points. The exceptional divisor is isomorphic to a cubic threefold, hence is covered by conics passing through these two points. We conclude that $b \geq c$. Finally, taking the elliptic curves with class $\lambda - 2e$ on the $K3$ fibers, we get the inequality $$2a \geq b + 3c.$$

We already know that $L, L-E$ and $2L-E-F$ are basepoint-free, hence they are movable. Assume that $3L-3E-F$ is movable and effective. If $a \geq b+c$, then $D$ is a nonnegative linear combination of basepoint-free divisors $$aL - bE -cF = c(2L-E-F) + (b-c)(L-E) + (a-b-c)L.$$ On the other hand, suppose that $b+c \geq a$. Then $D$ can be expressed as a nonnegative linear combination of movable classes as follows $$aL-bE-cF = (a-b)(2L-E-F) + (c-a+b)(3L-3E+ F) + (2a-b- 3c) (L-E).$$ Note that since $2a \geq b+3c$ and $b+c \geq a$, this is indeed a nonnegative linear combination of movable classes.

There remains to check that $3L-3E-F$ is a movable class. The linear system $|L-E|$ induces the projection of $Z$ to a quadric threefold $Q$. The image $B$ of the base curve $C$ under this map is a degree $14$ and genus $16$ curve that lies on $Q$. We claim that $B$ is cut out by equations of degree 3 on $Q$. First, observe that $B$ is contained in a unique quadric. The images of the complete intersection $K3$ surfaces containing $C$ are degree $6$ $K3$ surfaces that contain $B$. Such $K3$ surfaces are $(2,3)$ complete intersections and are contained in the unique quadric threefold $Q$. Consequently, if $B$ were contained in another quadric, $B$ would be a component of a curve of degree 12 contradicting that $B$ has degree 14. The Riemann-Roch Theorem implies that
$h^0(\OO_B(3))= 27$. Therefore, modulo the subspace of cubics divisible by the equation of $Q$, $B$ is contained in a $3$-dimensional subspace of cubic threefolds. Consider the $K3$ surface $S$ cut out by $Q$ and a general cubic containing $B$. The linear system of cubics $H^0(\PP^3, \OO_{\PP^3}(3) \otimes I_Z)$ on $S$ has $B$ as a fixed component together with the residual linear series $|\lambda - 2e|$ on $S$. Since the latter linear system is basepoint-free, we conclude that the cubics cut out $B$. Consequently, the base locus of $3L-3E-F$ is contained in the proper transform of the cone over $B$ and is a movable effective class. This concludes the proof of the theorem.
\end{proof}

\section{Infinite cases}\label{sect-infinite}
In this section we verify the movable cone conjecture for some varieties whose movable cone has infinitely many faces. We consider $\PP^3 \times \PP^3$, $\PP^1 \times \PP^1 \times \PP^1 \times \PP^1$ and the double cover of $\PP^2 \times \PP^2$ branched along a divisor of type $(2,2)$. 
\subsection{Picard number 3}

If $X$ has Picard number 3, then $N^1(X/\PP^n)$ is 2-dimensional. This simplifies things in two ways: cones in this space have exactly 2 extremal rays, and  the nef cones of SQMs of $X$ over $\PP^n$ are ``linearly ordered" inside the movable cone. Using these features, we will prove some general statements for $X$ of Picard rank 3 in this section.

\begin{theorem} \label{theorem-ratpoly}
Let $X$ be a smooth variety of Picard rank 3 such that $-K_X$ is semi-ample and gives a fibration $X \arrow \mathbf P^n$ with irreducible fibers. Moreover, assume that the group $\PsAut(X/\PP^n)$ is infinite. Let $X_\alpha$ be an SQM of $X$ obtained by a sequence of flops, and suppose $X_\alpha$ has no flipping contraction. Then the effective nef cone $\Nef{X_\alpha}$ is rational polyhedral, spanned by semi-ample divisors.
\end{theorem}
\begin{proof}
Since $X_\alpha$ is obtained from $X$ be a sequence of flops, the bundle $-K_X$ remains semi-ample on $X_\alpha$. The decomposition of the movable cone of $X$ into nef cones of SQMs therefore looks as in Figure \ref{nefconesA}. 

Passing to the relative cones, the relative nef cone $\nef{X_\alpha/\PP^n}$ is a subcone of $\mov{X/\PP^n}$. First we consider the case when $\nef{X_\alpha/\PP^n}$ contains a boundary ray of  $\mov{X/\PP^n}$. Now $\PsAut(X/\PP^n)$ is an infinite group acting on the 2-dimensional cone $\mov{X/\PP^n}$. Passing to a subgroup of index 2 if necessary, we can assume that the group is infinite cyclic and acts by orientation-preserving maps. A generator of this group acts as a $2 \times 2$ matrix, hence has 2 one-dimensional eigenspaces, which correspond to the boundary rays of $\mov{X/\PP^n}$. On the other hand, the action of $\PsAut(X/\PP^n)$ permutes the nef cones of SQMs inside the relative movable cone and preserves orientation, so it cannot fix a ray of any nef cone. This is a contradiction,  showing that this case is impossible.

Therefore it must be the case that $\nef{X_\alpha/\PP^n}$ is contained in the interior of $\mov{X/\PP^n}$, Choose effective divisors $D_1$ and $D_2$ whose classes are as in Figure \ref{relmov}: they are in the interior of the relative effective cone, hence are relatively big. By \cite[Theorem 1.2]{BCHM}, for sufficiently small positive numbers $\epsilon_i$ we can run the MMP for the divisors $K_X+\Delta+\epsilon_i D_i$ over $\PP^n$; in both cases this consists of finitely many flops. So the two edges of $\nef{X_\alpha/\PP^n}$ correspond to flopping contractions. 

So far we have proved that each SQM $X_\alpha$ has exactly 2 flopping contractions. These give 2 extremal rays $R_1$, $R_2$ of the cone of curves $\overline{\Curv}(X_\alpha)$ lying in the hyperplane $K_X^\perp$. Since $\rho(X)=3$, $K_X^\perp$ is a 2-dimensional vector space, and so these two rays are the only extremal rays of $\overline{\Curv}(X_\alpha)$ in $K_X^\perp$. If $\overline{\Curv}(X_\alpha)$ has infinitely many $K$-negative extremal rays $R_i$, then by the Cone Theorem they must converge to a ray $R$ in $K_X^\perp$ (although $R$ need not be extremal). On the other hand, since there are no flipping contractions, each $R_i$ must correspond to a divisorial contraction or fibration of $X$. So if we pick an effective divisor $D$ such that $D \cdot R_1 <0$ and $D \cdot R_2 <0$, then we can have at most finitely many $i$ for which $D \cdot R_i <0$. As in Figure \ref{accum}, we then see that the $R_i$ must converge to a ray in $K_X^\perp$ which is strictly outside the cone $\langle R_1, R_2 \rangle$. This contradicts the fact that $R_1$ and $R_2$ are extremal rays of $\overline{\Curv}(X_\alpha)$. Therefore $\overline{\Curv}(X_\alpha)$, and dually $\nef{X_\alpha}$ has only finitely many extremal rays.

We have proved that $\nef{X_\alpha}$ is a rational polyhedral cone. It remains to show that it is spanned by semi-ample divisors; in particular this will imply that $\nef{X_\alpha}=\Nef{X_\alpha}$ as claimed. One extremal ray of $\nef{X_\alpha}$ is spanned by $-K_X$, which is basepoint-free, so we need to prove the claim for the other extremal rays of the nef cone. To do this, we use the Basepoint-Free Theorem \cite[Theorem 3.3]{KollarMori}. That theorem says that if $D$ is a nef line bundle and $D-K_X$ is nef and big, then $D$ is semi-ample. Now let $D$ be a nef line bundle on $X_\alpha$, not a multiple of $-K_X$: looking at Figure \ref{nefconesA} again, we see that $D-K_X$ must lie somewhere in the interior of $\mov{X}$, and hence be big, as required. 
\end{proof}

\begin{figure}
\centering
\begin{tikzpicture}[scale=2]

\draw[dashed][pattern=dots] (0,0) -- (-1,1.7) -- (-0.5,2) --cycle;
\draw[very thick] (0,0) -- (-1,1.7)  [shift={(0.4,-0.2)}]   node{$X_1$};
\draw[dashed][pattern=dots] (0,0) -- (1,1.7) -- (0.5,2) --cycle;
\draw[very thick] (0,0) -- (1,1.7)  [shift={(-0.4,-0.2)}]   node{$X_2$};
\draw[very thick][fill=lightgray] (0,0) -- (-0.5,2) -- (0.5,2) -- cycle ;
\draw (0,1.5) node{$X$};

\draw[dotted] (-1.05,1.65) -- (-1.2,1.3);
\draw[dotted] (-0.7,1) -- (-0.8,0.8); 
\draw[dashed] (1,1.7) -- (0.5,2);
\draw[dotted] (1.05,1.65) -- (1.2,1.3);
\draw[dotted]  (0.7,1) -- (0.8,0.8);
\draw[very thick] (0.5,2) -- (0,0) [shift={(0,-0.2)}] node{$-K$};

\draw[dotted](1.35,0.45) -- (1.35,0.3);
\draw[dashed][pattern=dots] (1.35,0.5) -- (1.3,1.1) -- (0,0) -- cycle;
\draw[very thick] (0,0) -- (1.3,1.1) ;
\draw[very thick] (0,0) -- (1.35,0.5) [shift={(-0.3,0.1)}] node{$X_\alpha$};
\end{tikzpicture} \caption{Nef cones inside the movable cone}  \label{nefconesA}
\end{figure}
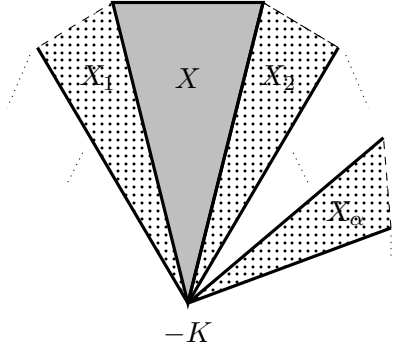

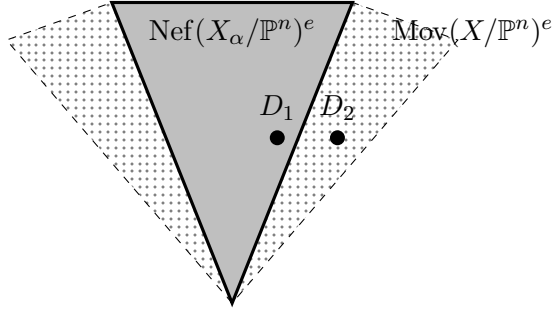
\begin{figure}
\centering
\begin{tikzpicture}[scale=2]

\draw[dashed][pattern=dots][pattern color=gray] (0,0) -- (-1.5,1.75) --(-0.8,2) -- (0.8,2) -- (1.5,1.75) --cycle;
\draw (1.6,1.8) node{$\Mov{X/\PP^n}$} ;
\draw[very thick][fill=lightgray] (0,0) -- (-0.8,2) --(0.8,2) --cycle;
\draw (0,1.8) node{$\Nef{X_\alpha/\PP^n}$};
\draw (0.3,1.1) node[circle,fill,inner sep=2pt,label=above:$D_1$]{};
\draw (0.7,1.1) node[circle,fill,inner sep=2pt,label=above:$D_2$]{};
\end{tikzpicture} \caption{Relative movable cone}  \label{relmov}
\end{figure}

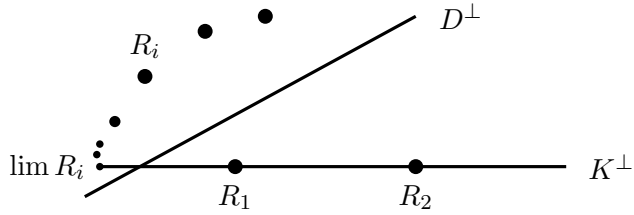
\begin{figure} 
\centering
\begin{tikzpicture}[scale=2]
\draw[very thick] (-1.11,0) -- (2,0)  [shift={(0.3,-0.0)}] node{$K^\perp$};
\draw[very thick] (-1.2,-0.2) -- (1,1)  [shift={(0.3,-0.0)}] node{$D^\perp$};
\draw (-0.2,0) node[circle,fill,inner sep=2pt,label=below:$R_1$]{};
\draw (1,0) node[circle,fill,inner sep=2pt,label=below:$R_2$]{};
\draw (0,1) node[circle,fill,inner sep=2pt]{};
\draw (-0.4,0.9) node[circle,fill,inner sep=2pt]{};
\draw (-0.8,0.6) node[circle,fill,inner sep=2pt,label=above:$R_i$]{};
\draw (-1.0,0.3) node[circle,fill,inner sep=1.5pt]{};
\draw (-1.1,0.15) node[circle,fill,inner sep=1pt]{};
\draw (-1.12,0.08) node[circle,fill,inner sep=1pt]{};
\draw (-1.1,0) node[circle,fill,inner sep=1pt,label=left:$\operatorname{lim} R_i$]{};
\end{tikzpicture} \caption{Accumulation of rays} \label{accum}
\end{figure}

\subsection{Branched double cover of $\PP^2 \times \PP^2$}
In this section we will use Theorem \ref{theorem-ratpoly} to deduce the Cone Conjecture for $X$ a very general blowup of the double cover of $\PP^2 \times \PP^2$ branched along a divisor of type $(2,2)$. The key point is to prove that $\PsAut(X/\PP^2)$ is infinite.

Let $Z$ be the double cover of $\PP^2 \times \PP^2$ branched along a divisor of type $(2,2)$. Then the Picard group of $Z$ is generated by $L_1, L_2$, the pullbacks of the hyperplane classes from the two projections to $\PP^2$, subject to the relations $$L_1^3 = L_2^3 = 0, \quad L_1^2 L_2^2 = 2.$$ Let $V$ be a two dimensional linear system in $|H|=| L_1 + L_2|$. Then the base locus of $V$ is a smooth curve of genus $7$. Let $X$ be the blowup of $Z$ along the base locus of $V$. Then $X$ admits a fibration  to $\PP^2$ whose generic fiber is a $K3$ surface $X_{\eta}$. The Picard group of $X_{\eta}$ is generated by $\lambda_1, \lambda_2$, the restrictions of $L_1$ and $L_2$, respectively.

We apply Corollary \ref{corollary-torellispecial} to show that $\Aut^*(X_\eta)$ is infinite, and hence that $\Nef{X_\eta}=P(X_\eta)$. We must show that the Picard group does not get bigger when we pass to the geometric generic fiber,  the base change of $X_\eta$ to the algebraic closure. We use the following lemma. A reference is \cite[Lemma 2.1]{Vial}; we give a proof here for convenience
 
\begin{lemma} \label{prop-steinitz}
Let $K$ denote the algebraic closure of the function field $\CC(t_1,t_2)$ of $\PP^2$, and let $X_{\overline{\eta}}$ denote the geometric generic fiber of $f$. Then there is a countable union of proper subvarieties $U = \cup_i U_i$ in $\PP^2$ such that for $p \in \PP^2 \setminus U$, the variety $X_{\overline{\eta}}$ is isomorphic as a scheme to the fiber $X_p = f^{-1}(p)$. In particular, for $p \in \PP^2 \setminus U$ we have an isomorphism $\Pic (X_p) \simeq \Pic (X_{\overline{\eta}})$. 
\end{lemma}

\begin{proof}
For simplicity we restrict to an affine patch $\bbA^2 \in \PP^2$ and take the $t_i$ to be coordinate functions on this affine space.

The key point of the proof is Steinitz' Theorem that two algebraically closed fields of the same uncountable cardinality are isomorphic. In particular, the field $K$ is isomorphic to the complex numbers $\CC$; moreover, the isomorphism can be chosen to identify any given transcendence bases of the two fields.  

Now let $k \subset \CC$ denote the (countable) subfield generated by all the coefficients of the divisors $D \in V$ in our linear system. Let $z_1,\ldots,z_2$ be complex numbers algebraically independent over $k$; note that the set of points $(z_1,z_2)$ for which this does not hold is a countable union $U= \cup_i U_i$ of proper subvarieties. Then there is  an isomorphism $\alpha: \CC \simeq K$ that fixes each element of $k$ and such that $\alpha^{-1}(t_i) = z_i$. Now consider the base change of $X_{\overline{\eta}}$ along $\alpha$: this yields the complex variety $\widetilde{X
} := X_{\overline{\eta}} \times_K \CC$. In terms of equations, $\widetilde{X}$ is obtained by replacing the indeterminates $t_i$ by the numbers $z_i$ in the equations of $X_{\overline{\eta}}$, and therefore is isomorphic to the fiber $X_p$ where $p=(z_1,z_2)$, as claimed.

The final statement follows since the  Picard group is (as an abelian group) determined up to isomorphism by the scheme structure, for example because it can be defined as the cohomology group $H^1(X,\mathcal O^\times)$. 
\end{proof}

By Ravindra--Srinivas \cite[Theorem 1]{RS}  a very general $K3$ surface $S$ obtained as a complete intersection of sections of $H$ has Picard rank 2. (Their theorem is stated in terms of divisors on 3-folds, but the Lefschetz hyperplane theorem applied to a smooth 3-fold  $V$ containing $S$ shows this is equivalent.) Combining this with the previous lemma gives the result we need: 
\begin{corollary} \label{corollary-nl}
Assume that the linear system $V \subset |H|$ defining our base locus is very general, meaning that the fiber of $f: X \arrow \PP^2$ over a very general point of $\PP^2 = V^*$ is a $K3$ surface of Picard number 2. Then the geometric generic fiber $X_{\overline{\eta}}$ has Picard number 2, and hence $\Pic(X_\eta)=\Pic(X_{\overline{\eta}})$.
\end{corollary}

\begin{proposition}
In the basis $\lambda_1, \lambda_2$, the intersection matrix is 
\begin{align*}
\begin{pmatrix}
 2 &  4 \\
 4 &  2
\end{pmatrix}
\end{align*}
The orthogonal group $O(\Pic(X))$ and the automorphism group $\Aut^*(X_\eta)$ are infinite. Hence $\Nef{X}=P(X)=\langle a\lambda_1 + b \lambda_2 \mid a = b \left(-2 \pm \sqrt{3} \right) \rangle$. 
\end{proposition} 

\begin{proof}
To calculate the intersection pairing, we compute
$$\lambda_1^2 = L_1^2(L_1 + L_2)^2 = 2, \quad \lambda_1 \lambda_2 = L_1 L_2(L_1 + L_2)^2 = 4, \quad \lambda_2^2 = L_2^2(L_1 + L_2)^2 = 2.$$
A $(-2)$-class $a \lambda_1 + b \lambda_2$ satisfies $2(a+b)^2 + 4 ab = -2.$ If $a+b$ is even, then the left hand side is divisible by $4$ whereas the right hand side is not. We conclude that $a+b$ must be odd. Without loss of generality, assume $a$ is even and $b$ is odd. Reducing the equation modulo 8, we see that the left hand side is $2$ and the right hand side is $-2$. We conclude that the equation has no solutions. Hence, $\Pic(X)$ does not have any $(-2)$-classes.

Let $\alpha: \Pic(X) \arrow \Pic(X)$ be the linear map determined by the matrix
\begin{align*}
\begin{pmatrix}
15 & 4 \\
-4 & -1
\end{pmatrix}
\end{align*}
Then it is an easy calculation to show that $\alpha \in O(\Pic(X))$; it has infinite order because, for example, its eigenvalues are not roots of unity. Since there are no $(-2)$-classes, Corollary \ref{corollary-torellispecial} gives the rest.
\end{proof}

Now we can complete the proof of the cone conjecture for $X$.

First observe that every SQM $X_\alpha$ of $X$ is obtained from $X$ by a sequence of flops by Theorem \ref{theorem-noflips}, and  every effective nef cone $\Nef{X_\alpha}$ is rational polyhedral by Theorem \ref{theorem-ratpoly}.

Now $\rho(Z)=2$, so $\rho(X/\PP^2)=2$, and so $\Nef{X/\PP^2}=\Nef{X_\eta}$ is a cone in a 2-dimensional vector space. Since $\PsAut^*(X/\PP^2) =\Aut^*(X_\eta)$ is infinite, we can find a rational polyhedral covering domain $P \subset \Nef{X_\eta}$. By running the MMP for two classes ``just outside" the two boundary rays of $P$, we see that there is a finite collection of SQMs $\{X_i\}_{i=1}^k$ such that $P$ is contained in the union of the relative nef  cones $\Nef{X_i/\PP^2}$. Now let $\Pi$ be the union of the absolute nef cones $\bigcup_i \Nef{X_i}$. Each nef cone is rational polyhedral, the class $-K_X$ is nef on each $X_i$ since they are obtained from $X$ by flops, and all their codimension-1 faces inside the big cone are dual to classes of $K$-trivial curves because no $X_i$ has a flipping contraction. So all the conditions of Theorem \ref{theorem-lifting} are fulfilled, and hence the cone conjecture is true for $X$.

\subsection{$\PP^3 \times \PP^3$}
In this section, we study the example of $Z = \PP^3 \times \PP^3$.

The Chow ring of $\PP^3 \times \PP^3$ is isomorphic to $$\ZZ [L_1, L_2]/ (L_1^4, L_2^4),$$ where $L_1$ and $L_2$ denote the pullback of the hyperplane classes via the two projections. Let $H=L_1 + L_2$ and let $V \subseteq |H|$ be a general $4$-dimensional linear system. Then the base locus $Bs(V)$ is a smooth curve $C$ with class $10L_1^2L_2^3 + 10 L_1^3 L_2^2$.  By adjunction, the genus of $C$ is $11$. Blowing up $C$, we get a smooth variety $X$ with a fibration $X \arrow \PP^4$ whose generic fiber $X_\eta$ is a $K3$ surface of class $4L_1^3L_2 + 6 L_1^2 L_2^2 + 4L_1L_2^3$. A basis for $\Pic(X_\eta)$ is the classes $\lambda_1$ and $\lambda_2$ which are the restrictions of $L_1$ and $L_2$, respectively. We denote the exceptional divisor of the blowup by $F$.

As in the previous case, using the analogue of Lemma \ref{prop-steinitz} and Noether--Lefschetz we again obtain:
\begin{proposition}
Assume the linear system $V \subset |H|$ is very general. Then $\Pic(X_\eta)=\Pic(X_{\overline{\eta}})$. 
\end{proposition}

\begin{proposition}
In the basis $\{\lambda_1,\lambda_2\}$, the intersection matrix on $\Pic(X_\eta)$ is
\begin{align*}
\begin{pmatrix}
4 & 6 \\
6 & 4
\end{pmatrix}
\end{align*}
There are no $(-2)$-classes. The groups $O(\Pic(X))$ and $\Aut^*(X_\eta)$ are infinite. Hence $\Nef{X}=P(X)=\langle a\lambda_1+b\lambda_2 \mid a~=~\frac{b}2 \left(-3 \pm \sqrt{5} \right) \rangle$.
\end{proposition}
\begin{proof}
We calculate the intersection matrix as follows:
\begin{align*}
\lambda_i^2 &= L_i^2 \cdot (L_1+L_2)^4 = 4 \quad (i=1,2);\\
\lambda_1\lambda_2 &= L_1L_2 \cdot (L_1+L_2)^4 =6
\end{align*}
since $L_i^4=0$ and $L_i^3L_j^3=1$ for $i \neq j$. 

If $a\lambda_1+b\lambda_2$ is a $(-2)$-class in $\Pic(X)$, then $(a,b)$ is an integer solution of 
\begin{align*}
4a^2+12ab+4b^2=-2.
\end{align*}
This is impossible since the left-hand side is divisible by 4 and the right hand side is not.

To see that $O(\Pic(X))$ is infinite, we can take for example the linear map $\alpha: \Pic(X) \arrow \Pic(X)$ determined by the matrix
\begin{align*}M&=
 \begin{pmatrix}
21 & 8 \\
-8 & -3
\end{pmatrix}.
\end{align*}
As before one checks this belongs to $O(\Pic(X))$, and has eigenvalues which are not roots of unity, so cannot have finite order. So $O(\Pic(X))$ is infinite. The rest follows from Corollary \ref{corollary-torellispecial} as before.
\end{proof}

%
%

Corollary \ref{corollary-torellispecial} gives us more precise information about automorphisms of $X_\eta$. It will tell us that a specific matrix, rather than some power of it, comes from an automorphism. This will imply that we only need to understand the nef cones of a small number of SQMs of $X$, rather than all of them as in the previous case. The statement we use is the following
\begin{proposition}
The lattice automorphism of $\Pic(X_\eta)$ with matrix
\begin{align*}
M= \begin{pmatrix}
21 & 8 \\
-8 & -3
\end{pmatrix}.
\end{align*}
acts by $- \Id$ on the discriminant group $L^*/L$, hence induces an automorphism $\phi_M$ of $X_{\eta}$. The cone $$\langle 3\lambda_1 - \lambda_2, 3 \lambda_2 - \lambda_1\rangle$$ is a fundamental domain for the action of $\phi_M$ on $\Nef{X_\eta}$.
\end{proposition}

\begin{proof}
By Corollary \ref{corollary-torellispecial} it is enough to show that $M$ satisfies the hypotheses of Theorem \ref{corollary-torelli2}.

The dual lattice $L^*$ is  $$\ZZ \left(\frac{1}{10}, \frac{1}{10}\right)^T \oplus \ZZ \left(\frac{3}{10}, \frac{-2}{10}\right)^T $$ The discriminant group $L^*/L$ is the order 20 group isomorphic to $\ZZ/2\ZZ \oplus \ZZ / 10 \ZZ$. An easy check shows that the action of $M$ on $L^*/L$ is $-\Id$, as required.

Finally, the matrix $M$ takes $(-1, 3)^T$ to $(3,-1)^T$ giving us the desired fundamental domain.
\end{proof}

By Theorem \ref{theorem-lifting}, in order to prove the cone conjecture for $X$, it is enough to cover the cone $$V=\langle 3\lambda_1 - \lambda_2, 3 \lambda_2 - \lambda_1\rangle$$ by nef cones of SQMs of $X$. In this example the geometry is somewhat simpler than in the previous one, and so we are able to calculate the nef cones of the SQMs we need explicitly.  

\begin{proposition} \label{nefconesofflops}
For $i=1,2$ let $X_i$ denote the SQM of $X$ obtained by flopping the locus $Z_i \subset X$ of curves with class $l_i-e$. Then 
\begin{align*}
\Nef{X_i} = \langle H_1+H_2-F, \, H_j, \,4H_j-F \rangle \quad (j \neq i).
\end{align*}
\end{proposition}
\begin{proof}
The above cone is cut out by the classes $l_i$, $f-l_i$, and $3l_i+l_j-4f$. We claim all these are classes of effective curves on $X_i$. The first class is clearly effective on $X_i$; the second is the class of a contracted line in the cocenter of the flop $X \dashrightarrow X_i$. 

To see the third class is effective, observe that it is the class on $X$ of the proper transform of an irreducible curve of bidegree $(3,1)$ that intersects the base curve 4 times. By Cayley's formula, a curve of degree 10 and genus 11 in $\PP^3$ has 20 4-secant lines counted with multiplicity \cite{GrusonPeskine}. Let $L \subset \PP^3$ be a 4-secant line of the projection of $C$ to the first factor, and let $S=\pi^{-1}(L)$. Then $S = \PP^1 \times \PP^3$ and $C$ intersects $S$ in $4$ points. We need to show that there exists a curve of type $(3,1)$ passing through the four points. By standard deformation theory, the dimension of the space of curves in $\PP^1 \times \PP^3$ in the class $\beta$ passing through $m$ general points is either empty or has dimension at least
$-K_{\PP^1 \times \PP^3} \cdot \beta + 1- 3m$. In our case, this dimension is 3 and the space is nonempty since there exists completely reducible curves passing through these 4 points. Furthermore, an easy dimension count shows that the general member of the space is irreducible. We conclude that there is an irreducible curve of class $3l_i+l_j$ on $\PP^3 \times \PP^3$ which is 4-secant to the base curve; its proper transform on $X_i$ is therefore a curve of class $3l_i+l_j-4f$ as required. Hence the above cone is an upper bound for $\Nef{X_i}$.

To see that it equals the nef cone, we need to prove that the three generators are actually nef and effective on $X_i$. For $H_1+H_2-F$ and $H_j$ this is clear, since they are basepoint-free on $X$ and have representatives disjoint from the flopping locus $Z_i$.
It remains to show that $4H_j-F$ becomes nef after the flop of $l_i-f$. In fact we can show more, namely that this line bundle becomes basepoint-free after we flop. 

To do this, we recall our setup. Let $Z_i$ be the locus of curves of class $l_i-f$. This is a codimension-2 subvariety of $X$. We have the projection morphism $\pi_i : X \arrow \PP^3$ which maps $Z_i$ to a smooth curve $\Gamma \subset \PP^3$. We also have the flopping contraction $f_i : X \arrow Y_i$ which contracts a curve if and only if it has class proportional to $l_i-f$, hence is birational with exceptional set $Z_i$. 

The flop $X \arrow X_i$ is obtained by taking $X_i = \operatorname{Proj} R(X,L_i)$ where $L_i$ is any line bundle whose proper transform on $X_i$ is ample. In particular we can take $L_i = f_i^*A + \epsilon (4H_j-F)$ where $A$ is an ample line bundle on $Y_i$ and $\epsilon >0$ is sufficiently small: since $f_i^*A$ lies in the interior of the common face of $\nef{X}$ and $\nef{X_i}$, this class $L_i$ must lie in the ample cone of $X_i$, as depicted in Figure \ref{fig2}. 

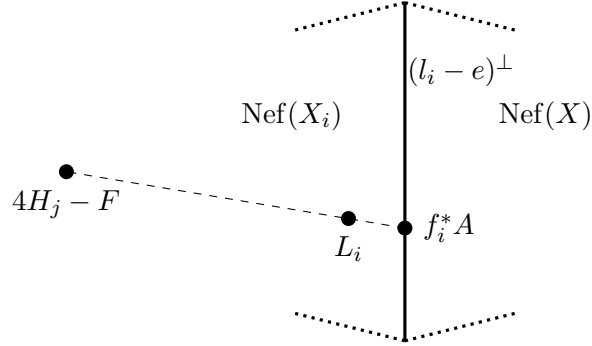
\begin{figure}
\begin{tikzpicture}[scale=0.75]
\draw[very thick] (0,-3) -- (0,3)  [shift={(1,-1.2)}] node{$(l_i-e)^\perp$};
\draw[very thick, dotted] (0,3) -- (2,2.5) [shift={(0.5,-1.5)}] node{$\nef{X}$};
\draw[very thick, dotted] (0,-3) -- (2,-2.5);
\draw[very thick, dotted] (0,3) -- (-2,2.5) [shift={(0,-1.5)}] node{$\nef{X_i}$};
\draw[very thick, dotted] (0,-3) -- (-2,-2.5);
\draw (-6,0) node[circle,fill,inner sep=2pt,label=below:$4H_j-F$]{};
\draw (0,-1) node[circle,fill,inner sep=2pt,label=right:$f_i^* A$]{};
\draw (-1,-0.83) node[circle,fill,inner sep=2pt,label=below:$L_i$]{};
\draw[dashed] (0,-1) -- (-6,0);

\end{tikzpicture} \caption{The nef cones of $X$ and $X_i$} \label{fig2}
\end{figure}

Now let $g_i : W_i \arrow X$ be the blowup of $X$ along $Z_i$ with exceptional locus $G_i$. Then
\begin{align*}
g_i^* L_i &= (f_i g_i)^* A + \epsilon g_i^* (4H_j -G_i) \\
&= (f_i g_i)^* A + \epsilon B + \epsilon G_i 
\end{align*} 
where $B$ is the class of the proper transform of the zero-set of a section of $4H_j-F$. Now zero-sets of sections of $4H_j-F$ are proper transforms on $X$ of pullbacks from $\PP^3$ of quartic surfaces containing the curve $\Gamma$. These surfaces cut out $\Gamma$ scheme-theoretically, so the base locus of the class $4H_j-F$ is precisely $Z_i$, and moreover divisors in this class have no common normal directions along $Z_i$. Hence when we blow up, the proper transform $B$ becomes basepoint-free. So we conclude that 
\begin{align*}
X_i &= \operatorname{Proj} R(X,L_i) = \operatorname{Proj} R(W_i,g_i^* L_i) = \operatorname{Proj} R(W_i,(f_i g_i)^* A + \epsilon B )
\end{align*}
where the second equality comes from the fact that throwing out the exceptional divisor $G_i$ does not affect $\operatorname{Proj}$. In other words we get $X_i$ by blowing up along $W_i$ and then taking the contraction $\varphi$ determined by the semi-ample line bundle $(f_i g_i)^* A + \epsilon B$. It remains to understand which curves are contracted by $\varphi$, and to show that $\varphi_* B$, which is the proper transform of $4H_j-F$, remains basepoint-free.

A curve $C \subset W_i$ is contracted by $\varphi$ if and only if $((f_i g_i)^* A) \cdot C = B \cdot C = 0$. Since $A$ is ample on the base of the flopping contraction, the first equality can only hold if $f_i g_i (C)$ is a point: in other words if $C$ is contracted by $g_i$ or if $g_i(C)$ is a multiple of a curve of class $l_i-f$. If $C$ is contracted by $g_i$ then $C \cdot G_i <0$, and hence $B \cdot C = g_i^*(4H_j-F) \cdot C - F \cdot C =0- G_i \cdot C>0$. So we can restrict attention to curves $C$ such that $g_i(C)$ is a multiple of $l_i-f$.  Now the fiber of $Z_i \arrow \Gamma$ over a point of $\Gamma$ is the blowup of $\PP^3$ at a point, and the curves of class $l_i$ are the lines through the origin. The normal bundle of $Z_i$ is just the pullback of the normal bundle of $\Gamma$, so it restricts to the trivial bundle on each of these fibers; therefore $G_i$ is the trivial $\PP^1$-bundle over $Z_i$. If $C \subset G_i$ is a curve such that $g_i(C)$ is a multiple of $l_i-f$, then it lives inside the surface $g^{-1} g_i(C) = \PP^1 \times \PP^1$; restricting $B$ to such a surface, we get a linear system of lines whose coordinate in the second $\PP^1$ is constant (specified by the normal direction of the corresponding quartic at the corresponding point of $\Gamma$). So a curve $C$ in such a surface with $B \cdot C=0$ must also be a line with constant coordinate in the second $\PP^1$. Contracting all such curves to get $X_1$, two disjoint sections of $B$ map to two distinct points. So $\varphi_* B$ remains basepoint-free on $X_1$, as required.
\end{proof}
It is now simple to deduce the main result of this section. For clarity, we repeat the assumptions on our linear system.
\begin{corollary} \label{P3P3main}
Let $Z=\PP^3 \times \PP^3$ and $H$ the line bundle $H_1+H_2$. Let $V \subset |H|$ be a very general linear system of dimension 4, with base locus $C$, and let $X$ be the blowup of $Z$ along $C$. Then the movable cone conjecture holds for $X$.
\end{corollary}
\begin{proof}
By the previous proposition, the quotient map $q : N^1(X) \arrow N^1(X/\PP^4)$ maps the rational polyhedral cone $\Nef{X} \cup \Nef{X_1} \cup \Nef{X_2}$ onto the covering domain $V = \langle 3\lambda_1 - \lambda_2, 3 \lambda_2 - \lambda_1\rangle$. Hence the conditions of Theorem \ref{theorem-lifting} are satisfied.
\end{proof}

We can pass from $Z$ to a general section without affecting the calculations above; this yields two more cases of the conjecture. Recall that $F(1,3;4)$ is isomorphic to a divisor of type $(1,1)$ in $\PP^3 \times \PP^3$. Let $Z=F(1,3;4)$ or the complete intersection of two $(1,1)$ divisors in $\PP^3 \times \PP^3$. Let $H$ be the restriction of $H_1+H_2$ to $Z$. Let $V \subset |H|$ be very general linear system of dimension $3$ or $2$, respectively. We conclude the following:
\begin{corollary}
The cone conjecture holds for $X$.
\end{corollary}
\begin{proof}
All the calculations of the previous case carry over unchanged to these two cases. Note that in these cases $X=f^{-1}(\Lambda)$ where $f: \operatorname{Bl}_C(\PP^3 \times \PP^3)  \arrow \PP^4$ is the fiber space from the previous case, and $\Lambda \subset \PP^4$ is a linear subspace of codimension 1 or 2. Since all the flops and pseudo-automorphisms in the previous case happen over the base, everything restricts to these subvarieties. In particular, in each case $X$ has two SQMs $X_1$ and $X_2$ obtained by flopping curves of class $l_1-f$ and $l_2-f$.

All that needs to be checked is that the nef cones from the previous example do indeed restrict to the nef cones of the SQMs $X$, $X_1$ and $X_2$. For $X$ this is proved in Theorem \ref{thm-nef}. For the $X_i$, we just need to verify that the same curve classes exist on $X$. For the classes $l_i$ and $f-l_i$ this is clear. For $3l_i+l+j-4f$, we have to ensure that the curves on $\operatorname{Bl}_C(\PP^3 \times \PP^3)$ in this class that we constructed in the previous case can be forced to lie on $X$. Taking proper transforms on $\PP^3 \times \PP^3$, such a curve $\Gamma$ has $\Gamma \cdot (H_1+H_2)=4$. However, the deformation theory argument above showed that such a curve could be chosen to pass through 5 assigned points in $S=\PP^1 \times \PP^3$. Taking 4 points to be  the intersection points $S \cap C$ with the base curve, and the fifth to be any point of $S \cap Z$, we get that $\Gamma$ must be contained in any $(1,1)$-divisor containing $Z$. Since $Z$ is cut out by $(1,1)$-divisors in both cases, this shows $\Gamma \subset Z$, as required.
\end{proof}

\subsection{$\PP^1 \times \PP^1 \times \PP^1 \times \PP^1$}

 Finally, we verify the movable cone conjecture for a variety of higher Picard rank. This example has two interesting differences from the others we have studied. First, the $K3$ surface occurring as the generic fiber of our fibration has Picard rank 4, and this allows for a more complicated non-abelian automorphism group. Second, we do not invoke the Global Torelli theorem to find automorphisms of the $K3$ surface: instead we obtain them from certain elliptic fibrations on the surface.

Our setup is as follows. Let $Z = \PP^1 \times \PP^1 \times \PP^1 \times \PP^1$. Take a general 2-dimensional linear subsystem $V \subseteq |H|$, where $H=-\frac12K_Z=\mathcal{O}(1,1,1,1)$. Then $Bs(V)$ is a smooth curve $C$ of multidegree $(6,6,6,6)$ and genus 13. Let $X$ be the blowup of $Z$ along $C$. So there is a morphism $\pi: X \rightarrow \PP^2$ with generic fiber a $K3$ surface $X_\eta$. 

We fix the following notation: for $i=1,2,3,4$ let $\pi_i: X \rightarrow \PP^1$ be the projections, and $H_i$ the pullback by $\pi_i$ of $\mathcal{O}_{\PP^1}(1)$. Let $F$ be the exceptional divisor of the blowup along $C$.

Now we study the movable cone of $X$. As usual, the first step to understand the nef cone of the surface $X_\eta$. First note that $\rho(X_\eta)=\rho(X)-1=4$, since there are no reducible members in the linear system $-\frac12K_X$. For each $i$, the line bundle $L_i = (H_i)_{|X_\eta}$ gives a basepoint-free linear system on $X_\eta$ with $L_i^2=0$: that is, each $L_i$ defines an elliptic fibration $p_i: X_\eta \rightarrow \PP^1$. Moreover, by calculating products of the form $H_i \cdot H_j \cdot (\sum_k H_k)^2$ in $CH^*(X)$, we find that the intersection matrix of $\Pic(X_\eta)$ with respect to the basis $L_i$ is 
\begin{align} \label{k3-int} \tag{$\ast$}
\begin{pmatrix}
0 & 2 & 2 & 2\\ 2 & 0 & 2 &2 \\ 2 &2 & 0 & 2 \\ 2 & 2 &2 & 0
\end{pmatrix}
\end{align}

\begin{proposition}
Fix any ample divisor $H$ on $X_\eta$. Then 
$\Nef{X_\eta}$ is the rational hull $P_+$ of the positive cone $P=\{x \in N^1(X_\eta) \mid x^2 > 0 \text{ and } x \cdot H >0\}$.  
\end{proposition}
\begin{proof}
For any $K3$ surface $S$, the nef cone $\nef{S}$ is defined inside the closed positive cone $\overline{P}$ by the hyperplanes dual to classes of $(-2)$-curves on $S$. Using the matrix (\ref{k3-int}) to compute the self-intersection of classes on $X_\eta$ we get
\begin{align*}
\left(\sum_i a_i L_i \right)^2 = 4 \sum_{i \neq j} a_i a_j
\end{align*}
so there are no $(-2)$-classes. 
\end{proof}

So all the extremal rays of $\Nef{X_\eta}$ correspond to elliptic pencils. There are infinitely many of these, because the intersection form defines a projective quadric threefold with a rational point, hence infinitely many.  This implies (for example by Sterk \cite{Sterk1985}) that the automorphism group of $X_\eta$ must be infinite. We get a supply of automorphisms from the four elliptic fibrations $p_i$ of $X_\eta$; let us now study these fibrations in more detail.


To fix notation let us consider the fibration $p_1$. Each of $L_2$, $L_3$, $L_4$ restricts to a divisor of degree 2 on $C_1$, so the divisors $L_i-L_j$ (for $i, j \in \{2,3,4\}$) have degree $0$. Translation by each of these elements gives an automorphism of $C_1$, hence of $X_\eta$. Let us denote these automorphisms by $\tau^1_{ij}$. 

Let us write down the matrix of $\tau^1_{32}$ acting on $\Pic(X_\eta)$ in the basis $\{L_1,\ldots,L_4\}$. We calculate this matrix by using the formula for the action of $\Pic^0(C_1)$ on $\Pic(C_1)$: an element $y \in Pic^0(C_1)$ acts by
\begin{align*}
x &\mapsto x + \operatorname{deg}x \cdot y 
\end{align*}
and so this extends to an action on $\Pic(X_\eta)$ by the formula
\begin{align*}
x &\mapsto x + (x \cdot C_i) \cdot y + m(x,y) F 
\end{align*}
where $F$ is the class of a fiber and $m(x,y)$ is an integer that we can calculate by using the fact that automorphism preserve self-intersection number. 

Applying this with $y=L_3-L_2$, and $x=L_i$ (for $i=1,\ldots,4$) we find that $\tau_{32}^1$ acts by the matrix
\begin{align*} M^1_{32} = 
\begin{pmatrix}
1 & 2 & 6 &4 \\
0 & -1 & -2 & -2 \\
0 & 2 & 3 &2 \\
0 & 0 & 0 & 1 
\end{pmatrix}
\end{align*} 

As well as translations, we also have automorphisms $h^i_j$ of $X_\eta$ coming from hyperelliptic involutions. Again to fix notation we consider the fibration $p_1$. There is a hyperelliptic involution of $C_1$ fixing $L_2$: computing the action of this involution on $\Pic(X_\eta)$ we get the matrix
\begin{align*} H^1_2 = 
\begin{pmatrix}
1 & 0 & 2 &2 \\
0 & 1 & 2 & 2 \\
0 & 0 & -1 &0 \\
0 & 0 & 0 & -1 
\end{pmatrix}
\end{align*}

By symmetry, if we consider the other elliptic fibrations $p_i$ and the automorphisms $\tau^i_{jk}$ and $h^i_j$ arising from them, we just get a conjugate of these matrices by the appropriate permutation matrix. Therefore the image of $\Aut(X_\eta)$ in $GL(N^1X(X_\eta))$ contains the group 
\begin{align*}
G &= \langle M^1_{32}, H^1_2, S_4 \rangle
\end{align*} 
where $S_4$ is the group of $4 \times 4$ permutation matrices.
\begin{proposition}
The group $G$ acts on $\Nef{X_\eta}$ with rational polyhedral covering domain
\begin{align*}
V = \langle L_i, \sum_k L_k -2 L_i   \rangle \quad (i=1\ldots,4 ) 
\end{align*}
\end{proposition}
\begin{proof}
The idea is similar to the construction of Dirichlet domains in hyperbolic geometry, although a cruder version suffices here. 

For any subset $S \subset G$ and any $x \in P = \operatorname{Int}(\Nef{X_\eta})$, define
$$ D(x,S) = \left\{ y \in \overline{P} \mid x \cdot y \leq x\cdot gy \ \forall \ g \in S \right\}.$$
This is a closed convex cone inside $\overline{P}$. We observe the following:
\begin{itemize}
\item If $S \subset T$ then $D(x,T) \subset D(x,S)$. In particular, $D(x,G) \subset D(x,S)$ for any subset $S \in G$
\item $D(x,G)$ is a fundamental domain for the action of $G$ on $P$;
\end{itemize}
Putting these together, we see that to prove the proposition it suffices to find a finite subset $S \subset G$ and a point $x \in P$ such that $D_x^S$ equals $V$.

To do this we use computer calculations. In more detail, we take the generating set $\{ M^1_{32}, H^1_2, S_4\}$ for $G$, and we define $G(k) \subset G$ to be the set of all words of length $k$ in these elements and their inverses. For small values of $k$, the computer algebra system Sage \cite{sage} is able to compute the cone $D^{G(k)}_x$ for a given $x \in P$. We find that with $x=L_1+L_2+L_3+L_4$ and $k=2$, we have $D^{G(2)}_x = V$ as required. For the Sage code used to compute the cone, see \cite{sagecode}.
\end{proof}

\begin{corollary}
The variety $X$ satisfies the conditions of Theorem \ref{theorem-lifting}. Hence the Cone Conjecture is true for $X$. 
\end{corollary}
\begin{proof}
Again we need to cover the cone $V$ above by nef cones of SQMs of $X$. To do this, let us decompose $V$ as 
\begin{align*}
V=V_0 \cup V_1 \cdots \cup V_4
\end{align*}
where the $V_i$ are simplicial cones defined as follows: $V_0$ is the simplicial cone spanned by the $L_i$, and for $i=1,\ldots,4$ the cone $V_i$ is the span of the face of $V$ opposite $L_i$ together with the vector $\sum_{k \neq i} L_k - L_i$. 

Now $V_0$ is covered by $\Nef{X}$. Let us show how to cover $V_i$. Flopping the locus of curves on $X$ of class $l_i-e$, we obtain an SQM $X_i$ of $X$. We claim that $\Nef{X_i}$ maps onto $V_i$.  To check this, first note that $\Nef{X_i}$ contains the classes $L_j$ for $j \neq i$, so it remains only to show that $\sum_{k \neq i} L_k - L_i$ is in the image of $\Nef{X_i}$. Indeed,  $\sum_{k \neq i} L_k - L_i$, is mapped onto by the class the class $D_i=2\sum_{k \neq i} H_i -F = -\frac{1}{2}K_X + (\sum_{k \neq i} H_k - H_i)$; to see this class is basepoint-free on $X_i$ we can use the same argument as in  Proposition \ref{nefconesofflops}.

Finally it remains to show that $\Nef{X_i}$ is rational polyhedral. In fact we claim that $\Nef{X_i}$ is simplicial, spanned by the basepoint-free classes $-\frac{1}{2}K_X$, $H_j \, (j \neq i)$, and $D_i$. All of these classes are basepoint-free on $X_i$, so this is a lower bound for the nef cone; to prove it equals the nef cone, we need to show that the dual cone is spanned by effective curves. It is straightforward to calculate that the dual cone is spanned by $l_i$, $f-l_i$, and $l_i+l_j-2f$ for $j\neq i$. The first class is clearly effective; the second is the class of a contracted curve in the cocenter of the flop. To see that the third type of class is effective, observe that it is the class of the proper transform a curve of type $(1,1)$ in a surface  $ S_{ij} \subset (\PP^1)^4$ (a section of the $(i,j)^{th}$ projection  $\pi_{ij} : (\PP^1)^4 \arrow \PP^1 \times \PP^1$) which intersects the base curve $C$ in two points (counted with multiplicity). To see that such curves exist, observe that any two points in $\PP^1 \times \PP^1$ are joined by a $(1,1)$ curve, so it is enough to find a surface $S_{ij}$ intersecting $C$ in at least two points. For this, notice that $C$ has genus 13, whereas if $\pi{kl}$ is any projection onto a factor $\PP^1 \times \PP^1$, then $\pi_{kl}(C)$ has bidegree $(6,6)$, hence by adjunction it has arithmetic genus $\frac{1}{2} \left( 4\cdot 6 + 4 \cdot 6 +2 \right) = 25$, and therefore must have singular points. Now let $\{k,l\} = \{ 1,2,3,4\} \setminus \{i,j\}$, let $p$ be any singular point of $\pi_{kl}(C)$, and let $S_{ij} = \pi_{kl}^{-1}(p)$; then $S_{ij}$ intersects $C$ in at least two points (counted with multiplicity).
\end{proof}

\bibliographystyle{plain}

\end{document}